\documentclass[a4paper,13pt]{article}
\usepackage{a4wide}
\usepackage[english]{babel}
\usepackage{amssymb}
\usepackage{amsmath}
\usepackage{amsfonts}
\usepackage{amsthm}
\usepackage[pdftex]{graphicx}
\usepackage{mathrsfs}
\usepackage{hyperref}
\hypersetup{colorlinks,citecolor=red,filecolor=purple,linkcolor=blue,urlcolor=black}
\usepackage[T1]{fontenc}
\usepackage{enumerate}
\usepackage{relsize}
\usepackage{mathtools}
\usepackage{accents} 
\usepackage{bbm}

\newcommand{\R}{\mathbb R}
\newcommand{\E}{\mathbb E}

\newcommand{\me}{\medskip \noindent}
\newcommand{\bi}{\bigskip \noindent}
\newcommand{\un}{\mbox{\large$\mathbbm{1}$}}
\newcommand{\e}{\mathsmaller{E}}
\newcommand{\NvDelta}{\pmb{\Xi}}

\def\gr{ \textcolor{green}}

\def\be{\begin{eqnarray}}
\def\ee{\end{eqnarray}}
\def\ben{\begin{eqnarray*}}
\def\een{\end{eqnarray*}}

\newcommand{\TT}{\mathbb{T}}

\newcommand{\T}{\mathbb{T}}

\newcommand{\ind}{{\bf 1}}

\newcommand{\nt}{{\widetilde{N}}}
\newcommand{\wX}{{\widetilde{X}}}
\newcommand{\hX}{{\widehat{X}}}
\newcommand{\wY}{{\widetilde{Y}}}
\newcommand{\hY}{{\widehat{Y}}}
\newcommand{\td}{{t\wedge T_\delta}}
\newcommand{\Td}{{t\wedge T_{n,\delta}}}

\newcommand{\CU}{{\mathbf{C}}_1}
\newcommand{\cb}{{\overline{\mathbf{c}}_1}}

\newcommand{\K}{{\mathbf{K}}}

\newcommand{\RR}{\mathbb{R}}

\newcommand{\PP}{\mathbb{P}}

\newcommand{\EE}{\mathbb{E}}

\newtheorem{prop}{Proposition}[section]

\newtheorem{lem}[prop]{Lemma}
\newtheorem{thm}[prop]{Theorem}
\newtheorem{rem}[prop]{Remark}

\newtheorem*{assumptionA2'}{Assumption A2'}

\title{Scaling limits of  bisexual Galton-Watson processes}

\author{Vincent Bansaye\thanks{CMAP, Ecole Polytechnique, CNRS, route de Saclay, 91128 Palaiseau Cedex-France; 
E-mail: \href{mailto:vincent.bansaye@polytechnique.edu}{\texttt{vincent.bansaye@polytechnique.edu} }},
Maria-Emilia Caballero\thanks{UNAM; E-mail: \href{mailto:mariaemica@gmail.com}{\texttt{mariaemica@gmail.com}}},  
Sylvie M\'el\'eard\thanks{CMAP, Ecole Polytechnique, CNRS, route de
Saclay, 91128 Palaiseau Cedex-France; E-mail: \href{mailto:sylvie.meleard@polytechnique.edu}
{\texttt{sylvie.meleard@polytechnique.edu}}}, Jaime San Mart\'in\thanks{CMM-DIM,  Universidad de Chile, 
UMI-CNRS 2807, BASAL AFB170001, Chile; E-mail: \href{mailto:jsanmart@dim.uchile.cl}
{\texttt{jsanmart@dim.uchile.cl}}}.}

\begin{document}

\maketitle

\begin{abstract}
Bisexual Galton-Watson processes are discrete Markov chains where reproduction events 
are due to mating of males and females. Owing to this interaction, the standard branching property of
Galton-Watson processes is lost. We prove tightness for conveniently  rescaled bisexual Galton-Watson processes,
based on recent techniques developed in [4]. We also identify the
possible limits of these rescaled processes as solutions of a stochastic system,  coupling  two equations  through singular coefficients in Poisson terms added to square roots as coefficients of Brownian motions.   Under some additional integrability assumptions, pathwise uniqueness 
of this limiting  system of stochastic differential equations and convergence of the rescaled processes are obtained.  Two examples 
corresponding to mutual fidelity are considered. 
\end{abstract}


\tableofcontents

\section{Introduction}
Galton-Watson  processes describe population dynamics
for clonal populations without interactions. Biological reasons  have led to generalize these processes to bisexual Galton-Watson processes modeling sexual  reproduction. The number of pairing of males and females in one generation  is then modeled by  a mating function which can have different forms depending on different reproduction strategies: monogamous or polygamous reproduction, fidelity or not, one dominant male, etc.  These bisexual Galton-Watson processes have been introduced by Daley \cite{Daley} and  studied in particular by  Alsmeyer and R\"osler \cite{AR1, AR2}, see \cite{A, Hull} for surveys. 

We are interested in the scaling limit of a  bisexual Galton-Watson process. We consider a population composed of females and males. The two subpopulations have  their own dynamics (clonal reproduction or intrinsic death) but can also interact through  the sexual reproduction. In the latter, the mating function plays a main role. This work extends a previous paper  \cite{BCM},  in which   a general method was proposed for investigating scaling limits of finite dimensional Markov chains to diffusions with jumps. This method was applied   to two one-dimensional cases in random environment.
In both cases the uniqueness of the  limiting one-dimensional  diffusion process 
was based on the works of  Fu and Li \cite{FL}, Dawson and Li \cite{DL} and Li and Pu \cite{LP}, where the authors generalized the well-known uniqueness result for  Feller diffusion,  with H\"older-$1/2$ regularity in the diffusion coefficient. 

In the present situation,  the two populations, females and males, are coupled by the mating,  which makes the problem  more difficult. We  use the general result developed in \cite{BCM} to prove tightness and identification of the scaling limits of the bisexual processes. The limiting values are solutions of a two-dimensional system of  coupled stochastic differential equations with jumps and non regular coefficients. The main novelty concerns the uniqueness of these limiting values. Indeed the coupling of the two equations  through singular coefficients in Poisson terms added to square roots as coefficients of Brownian motions raises a deep difficulty. We resolve the problem under an integrability condition on the jump measure which covers a large number of cases.

 \bi The bisexual Galton-Watson process $Z^N=( F^N, M^N)$ that we consider is defined as follows.  It is a Markov process  taking values in $\mathbb N^2$ and satisfying the following induction identity for $n\geq 0$,
\be
F^N_{n+1} &=&  F^N_{n}  + \sum_{p=1}^{F^N_n} \mathcal E^{f, N}_{n,p}+ \sum_{p=1}^{g_N(F^N_n, M^N_n)} L^{f,N}_{n,p}, \\
M^N_{n+1} &=&  M^N_{n} +\sum_{p=1}^{M^N_n} \mathcal E_{n,p}^{m,N}+ \sum_{p=1}^{g_N(F^N_n,M^N_n)} L_{n,p}^{m,N},
\ee
where  $N\in \mathbb{N}$ scales the population size and for each $N$, the family of random variables\\  $\{ M_0^N,F_0^N,\mathcal E_{n,p}^{f,N},   \mathcal E^{m,N}_{n,p}, (L_{n,p}^{f,N},L^{m,N}_{n,p}) : n,p \geq 1\} $  is mutually independent.
The random variables $(M_0^N,F_0^N)$ are integer-valued and    the random variables $(\mathcal E_{n,p}^{f,N},   \mathcal E^{m,N}_{n,p}, (L_{n,p}^{f,N},L^{m,N}_{n,p}))$  are identically distributed for  $n,p\geq 1$ and take values in $\{-1,0,1, 2, \ldots\}=\{-1,0\} \cup \mathbb N$. We  denote their distributions as follows:
$$ \mathcal E_{n,p}^{\bullet,N}\stackrel {d}{=}  \mathcal E^{\bullet,N}, \qquad \quad
(L_{n,p}^{f,N},L^{m,N}_{n,p})\stackrel {d}{=}(L^{f,N},L^{m,N}),$$
for $\bullet \in \{f,m\}$. 
The terms related to the random variables $ \mathcal E_{n,p}^{\bullet,N}$  may model either survival  without offsprings 
($\mathcal E^{\bullet,N}=0$) or death  without offsprings  ($\mathcal E^{\bullet,N}=-1$) or more complex event including  an asexual  clonal reproduction with several offsprings ($\mathcal E^{\bullet,N}\in \mathbb N$). The random variables $(L_{n,p}^{f,N},L^{m,N}_{n,p})$ model the sexual reproduction issued from mating. 
 
\me 
The class of  bisexual Galton-Watson process defined above
combines  the classical asexual Galton-Watson processes and the bisexual Galton-Watson processes introduced by Daley \cite{Daley}.
  
  \me  Our main result will be applied in two cases. The particular case
 where $ \mathcal E_{n,p}^{\bullet,N}$ are  Bernoulli random variables with values in $\{0,-1\}$, describes whether or not individuals survive in the next generation. A second interesting example  concerns the case where $ \mathcal E_{n,p}^{\bullet,N}$ are nul and 
 $(L_{n,p}^{f,N},L^{m,N}_{n,p})\stackrel{d}{=}-(1,1)+(L_+^{f,N},L_+^{m,N})$, with $L_+^{\bullet, N}\in \{0,1,\ldots\}$. This case can be  interpreted as the replacement of the mating pair of female and male  by a random number of  females  and males in the next generation, via sexual reproduction. The function $g_N$ counts the number of mating in one generation. One of the main example of  mating function is $g_N(y,z)=y\wedge z$ and we illustrate our results with this function. It counts the number of pairing of male and female when their number is given by $y$ and $z$. Modeling the effective sexual interaction by such a function corresponds to monogamous mating with mutual fidelity.

\bi In Section 2, we state our assumptions and the main results and we develop the  two applications. We prove the tightness and  identification  of the sequence of scaled sexual Galton-Watson processes in Section 3. We then conclude the proof of the convergence theorem by using the uniqueness result proved in Section 4. In the latter,  we prove a  uniqueness result in a slightly more general framework which could be applied to different situations. This is is the main difficulty of the paper.

\section{Main results and applications}

 \subsection{Assumptions and  statement of convergence}
 
 \me Let us  state the assumptions under which we  obtain our main result. These assumptions will be partially relaxed  in the next sections for  weaker results.

 \me The two first assumptions govern the scaling of the reproduction and death events.  
 
 \noindent We introduce  a truncation function $h:\mathbb R\rightarrow \mathbb R$, which  is a 
continuous bounded  function coinciding with Identity in a neighborhood of zero.  For convenience, we assume in this paper   that $h(u)=u$ for $u\in [-1,1]$ and as an example, one can consider $h(u)=(-1)\vee (u \wedge 1).$

\me {\bf Assumption A.} \emph{ We consider a non-negative sequence $v_N$ going to $+\infty$  and we assume that}
  
\me {\bf (A1) -} \emph{For $\bullet \in \{f,m\}$, there exist $\alpha_\bullet\in \R$,   $\sigma_\bullet\geq 0$ and a  measure $\nu_\bullet$   on $(0,\infty)$ satisfying}
$\ \int_0^{\infty} (1\wedge u^2) \, \nu_\bullet(du) < +\infty$, such that 

\me \be
&&  \lim_{N\to \infty }  v_NN\,\E(h(\mathcal E^{\bullet,N}/N
))= \alpha_\bullet; \   \ 
  \lim_{N\to \infty }  v_NN \, \E(h^2(\mathcal E^{\bullet,N}/N
))= \sigma_\bullet^2 + \int_{(0,\infty)}  h^2(u) \nu_{\bullet}(du) ; \nonumber \\ 
&&
\qquad \qquad \qquad  \lim_{N\to \infty }    v_NN \, \E({\phi}(\mathcal E^{\bullet,N}/N
))  =  \int_{0}^{\infty} \phi(u) \nu_\bullet(du) \label{condtripGW}
\ee
\emph{ for any $\phi:\R\rightarrow \R $ continuous bounded and null in a neighborhood of $0$}. 

  \bi {\bf (A2) -} \emph{For $\bullet,\star \in \{f,m\}$,  there exist  $\alpha^S_{\bullet} \in \R$ and $\sigma^S_{\bullet,\star}\in   \R^+$      and a measure $\nu_{S}$  on $[0,\infty)^2$ satisfying
$\ \int_{[0,\infty)^2} 1\wedge (u_1^2+u_2^2) \, \nu_{S}(du_1, du_2)<+\infty$, such that }
 
 \me
 \be
&&  \lim_{N\to \infty }  v_NN\,\E\Big(h\Big(L^{\bullet,N}/N\Big)\Big)= \alpha_{\bullet}^S,\nonumber \\
&&
 \lim_{N\to \infty }  v_NN \, \E\Big(h_{\bullet}h_{\star} \Big((L^{f,N},  L^{m,N})/N\Big) \Big)=  (\sigma_{\bullet,\star}^S)^2 + \int_{[0,\infty)^2}  h_{\bullet}h_{\star}(u_1,u_2)  \nu_{S}(du_1, du_2)\nonumber 
 \ee
 \emph{where $h_{f}(u_1,u_2)=h(u_1)$ and  $h_{m}(u_1,u_2)=h(u_2)$, and}
\be
&&
  \lim_{N\to \infty }    v_NN \, \E\Big(\phi\Big((L^{f,N}, L^{m,N})/N\Big)\Big)  =  \int_{[0,\infty)^2} \phi(u_{1}, u_{2}) \nu_{S}(du_{1}, du_{2}) \label{condtripGW}
\ee \emph{for any $\phi:\R^2\rightarrow \R$ continuous bounded and null in a neighborhood of $0$.}  \\

Assumption {\bf (A1)} yields the classical necessary and sufficient condition for convergence of rescaled (asexual) Galton-Watson processes, see for instance  Grimwall \cite{Grimwall}, Lamperti  \cite{Lamp2}, Bansaye-Simatos \cite{BS} or \cite {BCM}.  Assumption {\bf (A2)} provides a natural bisexual counterpart. 

\bi Let us now  introduce the assumptions on the mating function.

\me {\bf Assumption B.}  
\emph{There exists a   non-negative function $g$   on $\mathbb{R}_{+}^2$ such that  \\ 
{\bf (B1)-} The sequence of   mating functions    $g_N$ (defined on $\mathbb{N}^2$) uniformly converges  to $g$, as $N$ tends to infinity :
}
 \be
 \label{CVmating}
  \sup_{(y,z) \in (\mathbb{ N}/N)^2}  \bigg|{g_N(Ny,Nz)\over N} - g(y,z)\bigg|  \stackrel{N\rightarrow \infty}{\longrightarrow} 0.
 \ee
{\bf (B2)-} \emph{The function  $g$ is dominated by $y\wedge z$: there exists $a,b\geq 0$ such that for any 
 $y,z\geq 0$,}
\be
\label{g-exp}  g(y,z)\leq a(y\wedge z)+b.\ee

\noindent {\bf (B3)-} \emph{The function $g$ is locally Lipschitz and $g(y,0)=g(0,z)=0$ for all $y,z$.}

\smallskip \noindent {\bf (B4)-} \emph{The function $g$ satisfies the ellipticity assumption: for any positive $\delta,n$,}
\be
\label{hyp-g}
 \inf\{g(y,z):\, \delta\le y\le n,\delta\le z\le n\}>0.
\ee

 
 \bi As a main example, we have in mind the   monogamous mating with mutual fidelity for which $g_N(y,z) =g(y,z)= y\wedge z. $
We refer to \cite{A}  about  mating functions and their impact on population dynamics and to  \cite{AR1}  and \cite{AR2}  in the particular case of  promiscuous mating. Assumption
{\bf (B2)} is restrictive regarding the behavior of $g$ when $y$ or $z$ goes to infinity.  But it covers our main modeling motivations. Besides it  can certainly be relaxed before explosion time thanks to localization arguments to capture other mating functions.\\

 An additional moment assumption for the jump measure will be involved for pathwise uniqueness.
In this section for convenience we make the following first  moment assumption and refer to the next sections for comments and extensions.

\bi  {\bf Assumption C.}
\emph{We denote by $\nu$ the measure on $\RR^2_+$ given by
$$
\nu(du_1,du_2)=\nu_f(du_1)\delta_0(du_2)+\delta_0(du_1)\nu_m(du_2)+\nu_S(du_1,du_2).
$$
We assume that 
\be
\label{hyp-nu}
\int_{\mathbb{R}_{+}^2}  (u_1+u_2) \nu(du_1,du_2) <+\infty.
\ee}


\bi Let us now state our main result.

\begin{thm} 
\label{main}  Let us suppose  that Assumptions {\bf A},  {\bf B}, {\bf C} hold and that  the sequence  $Z_0^N/N$ converges 
weakly to $Z_0=(F_0,M_0)\in [0,\infty)^2$. Then, 
the sequence of processes $(Z^N_{[v_N.]}/N)_{N}$ converges in law in  $\mathbb D([0,\infty), [0,\infty)^2)$ to the unique strong  solution 
$Z=(F,M)$  of  
\be
\label{EDS-limit}
F_t &=& F_0+\int_0^{t} \alpha_f\,  F_sds + \int_0^{t} \sigma_f \sqrt{F_s}\, dB_s^f 
\nonumber \\
&& +\int_0^{t} \int_{[0,\infty)^2} {\bf 1}_{\theta\leq F_{s-}}h(u) \widetilde N^f(ds, d\theta, du)+\int_0^{t} 
\int_{[0,\infty)^2} {\bf 1}_{\theta\leq F_{s-}}(u-h(u))  N^f(ds, d\theta, du)  \nonumber \\
&&   + \int_0^{t}\alpha_f^S\, g(F_s,M_s) \, ds+  \int_0^{t}    \sqrt{g(F_s, M_s)} dB^{1}_s \nonumber\\
&&+\int_0^{t} \int_{[0,\infty)^3} {\bf 1}_{\theta\leq g(F_{s-}, M_{s-})}
h(u_1) \widetilde N^S(ds, d\theta, du_1,du_2)\nonumber \\
&&+ \int_0^t\int_{[0,\infty)^3}  {\bf 1}_{\theta\leq g(F_{s-},M_{s-})}(u_1 - h(u_1) ) N^S(ds,  d\theta,du_1,du_2),
\nonumber\\
M_t &=& M_0+\int_0^{t} \alpha_m M_sds 
+ \int_0^{t} \sigma_m \sqrt{M_s}\, dB_s^m\nonumber \\
&& +\int_0^{t} \int_{[0,\infty)^2} {\bf 1}_{\theta\leq M_{s-}}h(u) \widetilde N^m(ds, d\theta, du)
+\int_0^{t} \int_{[0,\infty)^2} {\bf 1}_{\theta\leq M_{s-}}(u-h(u))  N^m(ds, d\theta, du)  \nonumber \\
&&+ \int_0^{t} \alpha_m^S\, g(F_s,M_s) \, ds+\int_0^{t}  \sqrt{g(F_s, M_s)} dB^{2}_s      \nonumber \\
\label{system:main}
&&+\int_0^{t} \int_{[0,\infty)^3} {\bf 1}_{\theta\leq g(F_{s-}, M_{s-})}
h(u_{2}) \widetilde N^S(ds, d\theta, du_{1},du_{2})\nonumber \\
&& + \int_0^t\int_{[0,\infty)^3}  {\bf 1}_{\theta\leq g(F_{s-},M_{s-})}(u_2 - h(u_2))  N^S(ds,  d\theta,du_1,du_2),
\ee
 where  $B^{f}$ and $B^{m}$ are one-dimensional Brownian motions, 
$$\begin{pmatrix} B^1 \\ B^2 \end{pmatrix} = \begin{pmatrix}  \sqrt{(\sigma_{f}^S)^2-(\sigma_{f{m}}^S)^4/(\sigma_{m}^S)^2}& (\sigma_{fm}^S)^2/\sigma_{m}^S  \\  0   &  \sigma_{m}^S \end{pmatrix}.B ,$$
 $B$ is a two-dimensional Brownian motion, 
 $N^S$, $N^{f}$ and $N^{m}$  are  Poisson point measures on $[0,\infty)^4$ and $[0,\infty)^3$, respectively   with intensity measures  $dsd\theta \nu_{S}(du_1,du_2)$,   $ds d\theta\nu_1(du) $ and $dsd\theta\nu_{m}(du) $, $\widetilde{N}$ being the compensated measure of $N$ and all these processes are independent. 
 \end{thm}
 
 \me Note that in the previous statement, $\sigma^S_{\bullet,\bullet}$ is denoted by $\sigma^S_{\bullet}$ for convenience. Besides Assumption {\bf A} ensures that the quantity $(\sigma^S_f)^2 -  (\sigma^S_{fm})^4/(\sigma^S_m)^2 $ is positive. It can be deduced   from Cauchy-Schwarz inequality applied to $\E\Big(\chi_{\varepsilon}(L^{f,N})\chi_{\varepsilon}(L^{m,N})\Big)$ by choosing $\chi_{\varepsilon}=h(1-\phi_{\varepsilon})$  with $\phi_{\varepsilon}$ an even  continuous bounded function on $\mathbb{R}$  null in $[0,\varepsilon]$ and equal to $1$ in $[2\varepsilon,\infty)$. 
  
 \me Note also  that if instead of Assumption  {\bf C}, we only require
$
\int_{\mathbb{R}_{+}^2}  1\wedge (u_{1}+u_2) \, \nu(\text{d}u_1,\text{d}u_2) <+\infty,
$
then the convergence  hold before the explosion time $\mathbb T_e=\lim_{n\rightarrow \infty}\inf\{ t \geq 0 : F_t \geq n \ \text{or } M_t\geq n\}.$ Moreover the tightness and identification can be achieved under the optimal two-order moment condition : $\int_{\mathbb{R}_{+}^2}  1\wedge (u_{1}^2+u_{2}^2)  \, \nu(du_1,du_2) <+\infty.$
The proof of the uniqueness requires a stronger moment assumption close to $0$, which slightly extends the first moment condition.

\me Lastly, note that because of {\bf (B3)}, $(0,0)$ is an absorbing point and any solution issued from $\mathbb{R}_{+}^2$ stays in $\mathbb{R}_{+}^2$.

\me The proof of Theorem \ref{main} will be given in Section \ref{sec:proof-thm}. It  consists in first  proving the tightness and the identification of the limit under  Assumptions {\bf A} and {\bf (B1)},  {\bf (B2)}, using a suitable functional space which exploits the independence of random variables. The uniqueness  of the limit will be proved with  additional Assumptions {\bf (B3)--(B4)} and {\bf C}. This uniqueness result is the main point of the paper. Indeed, the stochastic system given in \eqref{EDS-limit} is a true coupled system (the coupling being due to the mating), with radical diffusion coefficients and accumulating  jumps. 
At the best of our knowledge,  it is the first result of this type in the case of non polynomial coefficients. For the polynomial case, we  mention the  general approach    for the study of the martingale problem developed in e.g. \cite{CLS, FiLa}. Our   uniqueness  result is  stated and proved in a general framework in Section 4. 

\subsection{Applications}

We consider now two examples  for which we apply the previous result.
 For sake of clarity, we focus on the classical mating function
$$g(y,z)=y\wedge z.$$

\subsubsection{Survival and sexual reproduction}

In this first application, the probability for a given mating to leave one offspring or more in the next time step (generation)
 is low. But a large number of offsprings may be produced in a single mating. This random integer number  is denoted by $D^N$.
  The sex is  determined independently for each offspring: each new born is  a female  (resp. a male) with probability $q\in (0,1)$ (resp. $1-q$). 
Besides, we fix   $p_f,p_m\geq 0$ to determine   the death  probability of males and females in each generation.\\

We make the following assumption.

 \me {\bf Assumption D.} \emph{Let us  consider  $\alpha\in [0,\infty)$ and  a measure $\mu$  on $[0,\infty)$ 
 such  that}
\be
\label{fin-mu}
\int_{0}^{\infty}  u \mu(du)<\infty.\ee
\emph{We also consider   a  sequence   $(v_N)_N$ of positive real numbers tending to infinity and a truncation function $h$ and assume that }
$$
  \lim_{N\to \infty }  v_NN\,\E\left(h\left(\frac{D^N}{N}\right)\right)=\alpha+ \int_{0}^{\infty} h(u) \, \mu (du); \   \  \lim_{N\to \infty}    v_NN \, \E\left({\phi}\left(\frac{D^N}{N}
\right)\right)  =  \int_{0}^{\infty} \phi(u) \, \mu(du).$$
\emph{for any $\phi$ continuous bounded and null in a neighborhood of $0$. }

\bi
We observe that the choice of the truncation function $h$ does not impact
 $(\alpha,\mu)$. Moreover, given a  triplet  $(\mu, \alpha,(v_N)_N)$ as previously,  the  sequence  of random variables 
$(D^N)_N$ satisfying
  $$ \  P(D^N=1)=\frac{\alpha}{v_N},\qquad \forall u \in  [2/N,\infty) , \quad \mathbb P(D^N\geq Nu)=\frac{\mu[u,\infty)}{Nv_N}$$
satisfies Assumption {\bf D}.


\bi We now  give  by induction a formal definition of  the bisexual Galton-Watson $Z^N=( F^N, M^N)$ with sexual reproduction $D^N$, sex ratio  $q$ and death rates $(p_f,p_m)$.  
Given $(F^N_0,M^N_0)\in  \mathbb{N}^2$, we define for $n\geq 0$, 
\be
F^N_{n+1} &=& F^N_n+  \sum_{p=1}^{F^N_n} \mathcal E^{f, N}_{n,p}+ \sum_{p=1}^{M^N_n\wedge  F^N_n} L^{f,N}_{n,p}, \\
M^N_{n+1} &=& M^N_n+ \sum_{p=1}^{M^N_n} \mathcal E_{n,p}^{m,N}+ \sum_{p=1}^{M^N_n\wedge F^N_n} L_{n,p}^{m,N}, \label{scalesurv}
\ee
where for each $N$, the family of random variables  $\{M_0^N,F_0^N$, $\mathcal E_{n,p}^{f,N},   \mathcal E^{m,N}_{n,p}, (L_{n,p}^{f,N},L^{m,N}_{n,p}) : n,p\geq 1\}$ is mutually independent. Moreover
for each $n,p\geq 1$ and $\bullet \in \{f,m\}$, 
 $$\PP(\mathcal E_{n,p}^{\bullet,N}=-1)=1-\PP(\mathcal E_{n,p}^{\bullet,N}=0)=p_\bullet/v_N$$
corresponds  to the probability of death of each female and each male, while
$$(L_{n,p}^{f,N}, L^{m,N}_{n,p}) \stackrel{d}{=}  (L^{f,N}, L^{m,N})=\sum_{j=1}^{D^N}  \left( \mathcal B_j, 1- \mathcal B_j \right)$$
describes  the sex repartition of offsprings, 
where $(\mathcal B_j)_{j\geq 1}$ are independent  Bernoulli random variables with parameter $q$ independent of $D^N$.

\begin{thm} \label{Appliuno}
Under the  weak convergence of $Z_0^N/N$ to $Z_0=(F_0,M_0)\in [0,\infty)^2$ and Assumption {\bf D}, the sequence of processes $(Z^N_{[v_N.]}/N)_N$ converges  in law in  $\mathbb D([0,\infty), [0,\infty)^2)$  to the unique   strong solution $Z=(F,M)$  of  the following SDE :
\be
F_t &=& F_0-\int_0^{t} p_f F_sds +\alpha q \int_0^{t} (F_s \wedge M_s) \, ds +\int_0^t\int_{[0,\infty)^2}  {\bf 1}_{\theta\leq F_{s-} \wedge M_{s-}} qu  N(ds,  d\theta, du),
\label{App1} \\
M_t &=& 
M_0-\int_0^{t} p_m M_sds +\alpha(1-q) \int_0^{t} (F_s \wedge M_s) \, ds
+ \int_0^t\int_{[0,\infty)^2}  {\bf 1}_{\theta\leq F_{s-} \wedge M_{s-}} (1-q)u  N(ds,  d\theta, du),\nonumber
\ee where $N$ is a Poisson point measure on 
$[0,\infty)^3$ with intensity measure $ ds d\theta \mu(du)$.
\end{thm}


\me To apply Theorem \ref{main}, the technical point  to check is Assumption {\bf (A2)}.
It is deduced from the next  lemma.

\begin{lem} For any integers $(k,\ell)\ne(0,0)$,
$$Nv_N\E\left(1-\exp\left(-k\frac{L^{f,N}}{N}-\ell  \frac{L^{m,N}}{N} \right)\right)\stackrel{N\rightarrow \infty}{\longrightarrow} a_{k,\ell} \, \alpha
+  \int_{0}^{\infty}(1-e^{-a_{k,\ell}u})\mu(du),$$
where $a_{k,\ell}=kq+\ell (1-q).$
\end{lem}
\begin{proof}
By independence of the random variables $\mathcal B_j$ and conditioning by $D^N$,
$$\E\left(1-e^{-kL^{f,N}/N-\ell  L^{m,N}/N }\right)=1-\E\left(\left[qe^{ -k/N}+(1-q)e^{-\ell/N}\right]^{D^N}\right)=\E\left(f_{a_{k,\ell}^N}(D^N/N)\right),$$
where $f_{a}(x)=1-\exp(-ax)$  and
$$a_{k,\ell}^N= - N\log\left(qe^{-k/N}+ (1-q )e^{-\ell/N} \right). $$ 
Letting $N\rightarrow \infty$ and noticing that $a_{k,\ell}^N\rightarrow a_{k,\ell}>0$, we prove that $Nv_N\E\left(f_{a_{k,\ell}}(D^N/N)\right)\rightarrow
a_{k,\ell}  \alpha
+  \int_{0}^{\infty} f_{a_{k,\ell}}\, d\mu$
  by Assumption {\bf D}  and  conclude.
More precisely,  let us use a family of non-negative continuous bounded functions $\varphi_{\varepsilon}: [0,\infty)\rightarrow [0,1]$, which are equal to zero in 
$[0, \varepsilon]$ and equal to $1$ in $ [2\varepsilon, \infty)$.   The decomposition
$f_a=ah+(\varphi_\varepsilon+1-\varphi_\varepsilon)(f_a-ah)$ yields
\ben
Nv_N \E\left(f_{a_{k,\ell}^N}(D^N/N)\right)&=&a_{k,\ell}^N.Nv_N\E(h(D^N/N))+Nv_N\E\left(\varphi_{\varepsilon}(f_{a_{k,\ell}^N}-a_{k,\ell}^Nh)(D^N/N)\right)\\
&&\qquad \qquad \quad +Nv_N\E\left((1-\varphi_{\varepsilon})(f_{a_{k,\ell}^N}-a_{k,\ell}^Nh)(D^N/N)\right).
\een
By  Assumption {\bf D}, the first term converges to $a_{k,\ell}(\alpha+\int h \, d\mu)$ as $N$ tends to infinity.\\
The last term vanishes as $\varepsilon$ tends to $0$. To see that, we use  that there exists $C>0$ such that
 for $\varepsilon$ small enough and $a$ in a bounded set,  $\vert (1-\varphi_{\varepsilon})( f_{a}-ah)\vert (x) \leq C \varepsilon h(x)$ for any
  $x\geq 0$ and
$$v_NN\big\vert \E\left((1-\varphi_{\varepsilon})(f_{a}-ah)(D^N/N)\right)\big\vert  \leq C \varepsilon v_NN\E(h(D^N/N)),$$
while $v_NN\E(h(D^N/N))$ is bounded by Assumption {\bf D}.\\
The facts that the sequence of functions $(f_{a_{k,\ell}^N}-a_{k,\ell}^Nh)(x)-(f_{a_{k,\ell}}-a_{k,\ell}h)(x)$
tends uniformly to $0$ on the interval $[\varepsilon,\infty)$ as $N$ tends to infinity and 
that $\left(Nv_N\E\left(\varphi_{\varepsilon}(D^N/N)\right)\right)_N$ is bounded by the last part of Assumption {\bf D}
ensures that 
$$Nv_N\left\{\E\left(\varphi_{\varepsilon}(f_{a_{k,\ell}^N}-a_{k,\ell}^Nh)(D^N/N)\right)-
\E\left(\varphi_{\varepsilon}(f_{a_{k,\ell}}-a_{k,\ell}h)(D^N/N)\right)\right\}\stackrel{N\rightarrow\infty}{\longrightarrow}0,$$
 for any $\varepsilon>0$. We    conclude using the convergence of $Nv_N
\E\left(\varphi_{\varepsilon}(f_{a_{k,\ell}}-a_{k,\ell}h)(D^N/N)\right)$ to 
$\,\int \varphi_{\varepsilon} (f_{a_{k,\ell}}-a_{k,\ell}h) \,d\mu\,$ which also comes from
Assumption {\bf D}.
\end{proof}

\begin{proof}[Proof of Theorem \ref{Appliuno}]
The previous lemma ensures via an approximation argument relying on Stone-Weierstrass local theorem  (see \cite{BCM} for details), that 
 \be
&&  \lim_{N\to \infty }  v_NN\,\E\Big(h\Big(L^{\bullet,N}/N\Big)\Big)= \alpha_{\bullet}^S; \nonumber \\
&& \lim_{N\to \infty }  v_NN \, \E\Big(h_{\bullet}h_{\star} \Big((L^{f,N},  L^{m,N})/N\Big) \Big)=  (\sigma_{\bullet,\star}^S)^2 + \int_{\R_{+}^2}  h_{\bullet}h_{\star}(u_1,u_2)  \nu_{S}(du_1, du_2)\nonumber \\
&&
  \lim_{N\to \infty }    v_NN \, \E\Big(\phi\Big((L^{f,N}, L^{m,N})/N\Big)\Big)  =  \int_{[0,\infty)^2} \phi(u_{1}, u_{2}) \nu_{S}(du_{1}, du_{2}), \nonumber
\ee 
where  we recall that  $h_{f}(u_1,u_2)=h(u_1)$ and  $h_{m}(u_1,u_2)=h(u_2)$ and
$\phi$ is continuous bounded and null in neighborhood of $0$ and  where we set
\be \alpha_{f}^S&=&\alpha q+\int_0^{\infty} h(qu)\mu(du),\quad  \alpha_{m}^S=\alpha(1-q)+\int_0^{\infty} h((1-q)u)\mu(du), \nonumber \\
 \nu^S(A)&=&\int_{0}^{\infty} 1_{(qu,(1-q)u)\in A} \, \mu(du).  \nonumber
\ee

Assumption {\bf (A1)} is obviously satisfied :
\be
&&  \lim_{N\to \infty }  v_NN\,\E(h(\mathcal E^{\bullet,N}/N
))= \lim_{N\to \infty }  -v_NN\PP(\mathcal E^{\bullet,N}=-1)/N
=  -p_\bullet; \nonumber \\
&& \lim_{N\to \infty }  v_NN\,\E(h^2(\mathcal E^{\bullet,N}/N))=
  \lim_{N\to \infty }  v_NN \,\PP(\mathcal E^{\bullet,N}=-1)/N^2
=0; \nonumber \\ 
&&
 \lim_{N\to \infty }    v_NN \, \E({\phi}(\mathcal E^{\bullet,N}/N
))  =  0. \nonumber
\ee

\me Assumption {\bf B} is  guaranteed by our choice of mating function $x\wedge y$ and Assumption {\bf C}  comes from \eqref{fin-mu}. We can apply Theorem 2.1 and conclude.
\end{proof}

\subsubsection{Replacement of couples}

We assume that for each $N$,  $\mathcal E^{\bullet,N}=0$. Besides, the  reproduction random variables
$L^{f,N}$ and $L^{m,N}$ are independent  random variables taking values in $\{-1,0,1,\ldots\}$ and the marginal laws satisfy the following scaling assumption.

 \me {\bf Assumption E.} \emph{We consider  two triplets $(\alpha_{\bullet},\sigma_{\bullet}, \nu_{\bullet})$  for $\bullet \in \{f,m\}$ with the  conditions}
$$\alpha_{\bullet} \in \R; \qquad  \sigma_{\bullet} \geq 0; \qquad \int_0^{\infty} u\nu_{\bullet}(du) <\infty.$$
\emph{We consider  also a truncation function $h$ and a  non-negative sequence $v_N$ going to $+\infty$.  Let us assume that for $\bullet\in\{f,m\}$,
 \be
&&  \lim_{N\to \infty }  v_NN\,\E(h(L^{\bullet,N}/N
))= \alpha_{\bullet}; \quad
  \lim_{N\to \infty }  v_NN \, \E(h^2(L^{\bullet,N}/N
))=\sigma_{\bullet}+\int_{0}^{\infty} h^2(u) \, \nu_{\bullet}(du); \nonumber \\
&&
\qquad \qquad \qquad  \lim_{N\to \infty }    v_NN \, \E(\varphi(L^{\bullet,N}/N
))  =  \int_{0}^{\infty} \varphi(u) \, \nu_{\bullet}(du).  \nonumber  \label{condtripGW}
\ee
for any $\varphi$ continuous bounded and null in a neighborhood of $0$.}

\me We know from the historical study of   Galton-Watson processes that for any  such triplet $(\alpha, \sigma, \nu)$, there exist $(v_{N})_{N}$ and $(L^N)_{N}$ satisfying Assumption {\bf E}, see \cite{Kallenberg,  JS, BS}.

\me We consider for each $N\geq 1$ the following Markov chain where every pair dies after reproduction and leaves  independently  a random  number of males and females, independent from each other and distributed as  $(L^{f,N},L^{m,N})$.
It is defined by 
\be
F^N_{n+1} &=& F^N_n+  \sum_{p=1}^{M^N_n\wedge  F^N_n} L^{f,N}_{n,p}, \nonumber \\
M^N_{n+1} &=& M^N_n+ \sum_{p=1}^{M^N_n \wedge F^N_n} L_{n,p}^{m,N} , \nonumber 
\ee
where  $(L^{f,N}_{n,p}, L_{n,p}^{m,N} : n\geq 0, p\geq 1)$ are independent and  distributed as  $(L^{f,N},L^{m,N})$. Writing $(L^{f,N},L^{m,N})=-(1,1)+(L_+^{f,N},L_+^{m,N})$, it means that the pairs disappear in the next generation and are replaced by a number of males and females given by $L_+^{\bullet,N} \in \{0,1, \ldots\}$.  \\
Assumption {\bf E}  and the independence of $L^{f,N}$ and $L^{m,N}$   make Assumptions {\bf A}   and {\bf C} easy to check, while Assumption  {\bf B} is again a direct consequence of the choice of $g_N$.
We obtain

\begin{thm}
 Under the  weak convergence of $(Z_0^N/N)_{N}$ to $Z_0=(F_0,M_0)\in [0,\infty)^2$ and Assumption {\bf E}, the sequence of processes $(Z^N_{[v_N.]}/N)_N$ converges in law in $\mathbb D([0,\infty), [0,\infty)^2)$  to the unique strong  solution $Z=(F,M)$ of the stochastic differential equation  
\be
F_t &=& F_0 + \alpha_f \int_0^{t} F_s\wedge M_s \, ds +   \sigma_f \int_0^{t}  \sqrt{ F_s \wedge M_s} dB^f_s  \nonumber  \\
&&  +\int_0^{t} \int_{[0,\infty)^2} {\bf 1}_{\theta\leq F_{s-} \wedge M_{s-}}h(u) \widetilde N^f(ds, du, d\theta)  
+ \int_0^t\int_{[0,\infty)^2}  {\bf 1}_{\theta\leq F_{s-} \wedge M_{s-}}(u - h(u))  N^f(ds, du, d\theta),  \nonumber \\
M_t &=& M_0+\alpha_m \int_0^{t} F_s\wedge M_s \, ds +  \sigma_m \int_0^{t}  \sqrt{F_s\wedge M_s} dB^m_s  \nonumber  \\
&&+\int_0^{t} \int_{[0,\infty)^2} {\bf 1}_{\theta\leq F_{s-} \wedge M_{s-}}h(u) \widetilde N^m(ds, du, d\theta) + \int_0^t\int_{[0,\infty)^2}  {\bf 1}_{\theta\leq F_{s-} \wedge M_{s-}}(u - h(u))  N^m(ds, du, d\theta),  \nonumber 
\ee
where $B^f,B^m, N^f, N^m$ are independent, $B^f$ and $B^m$ are  Brownian motions, $N^f$ and $N^m$ are  Poisson point measures on $\mathbb{R}_{+}^3$, respectively   with intensity measures $ds du \nu_f(d\theta)$ and $ds du \nu_m(d\theta)$.
\end{thm}
The assumption $\int_0^{\infty}  u \, \nu_{\bullet}(du) <\infty$ 
guarantees both non-explosion and pathwise uniqueness. For tightness and identification, we just need
$\int_0^{\infty} (u^2\wedge 1)\, \nu_{\bullet}(du) <\infty$ for $\bullet\in \{f,m\}$,
while $\int_0^{\infty} (u\wedge 1)\, \nu_{\bullet}(du) <\infty$ is sufficient for pathwise uniqueness before the explosion time.
\begin{proof} We have
$$\E\Big(h\Big(L^{f,N}/N\Big)h\Big(L^{m,N}/N\Big)\Big)=\E\Big(h\Big(L^{f,N}/N\Big)\Big)\E\Big(h\Big(L^{m,N}/N\Big)\Big)$$ is of order of magnitude of $(1/v_NN)^2$ so that
$$ \lim_{N\to \infty }  v_NN \, \E\Big(h\Big(L^{f,N}/N\Big)h\Big(L^{m,N}/N\Big) \Big)=  0.$$
Similarly for $\phi(u_1,u_2)=\phi_{f,1}(u_1)+\phi_{m,1}(u_2)+\sum_{k=2}^K \phi_{f,k}(u_1)\phi_{m,k}(u_2)$ and $\phi_{\bullet,  k}$ continuous bounded and equal to zero in a neighborhood of zero, we have 
$$ \lim_{N\to \infty }    v_NN \, \E\Big(\phi\Big((L^{f,N}, L^{m,N})/N\Big)\Big) =\int_0^{+\infty} \phi_{f,1}(u_1)\nu_f(du_1)+\int_0^{+\infty} \phi_{m,1}(u_2)\nu_m(du_2).$$
Thus, Assumption ${\bf A}$ holds with  $\sigma^S_{fm}=0$ and
 $\nu_S(du_1,du_2)=\delta_0(du_1)\nu_m(du_2)+\nu_f(du_1)\delta_0(du_2)$
and applying Theorem \ref{main} yields the result.
\end{proof}

 \section{Proof of the convergence}
 \label{sec:proof-thm}
 
  The proof is organized as follows. First, using \cite{BCM} applied to  a compactified version of the bisexual process $Z^N=(F^N,M^N)$, we  prove  tightness and that the limiting points of $Z^N$ are weak solution of  SDE \eqref{EDS-limit}. 
Second, we prove that pathwise uniqueness  holds for  \eqref{EDS-limit}. This point is new and is the main difficulty  of the paper. It  is  the object of forthcoming Proposition \ref{the:main2}, whose proof is a direct adaptation of the uniqueness result stated and proved in a more convenient setting in Section \ref{sec:uniqueness}.
 
 \subsection{Tightness and identification}
 
Tightness and identification are proved under more general assumptions.
We only  need Assumptions {\bf A} and {\bf (B1)}, {\bf (B2)}. 
 \begin{prop}
\label{prop=tightness}  Suppose  Assumptions {\bf A} and  {\bf (B1)}, {\bf (B2)} hold and suppose the sequence  $(Z_0^N/N)_{N}$ converges 
weakly to $Z_0=(F_0,M_0)\in [0,\infty)^2$. Then, 
the sequence of processes $(Z^N_{[v_N.]}/N)_{N}$ is tight in   $\mathbb D([0,\infty), [0,\infty]^2)$ and the limiting values 
$Z=(F,M)$ are weak solutions of  \eqref{EDS-limit} before the explosion time $\mathbb T_e=\lim_{n\rightarrow \infty}\inf\{ t \geq 0 : F_t \geq n \ \text{or } M_t\geq n\}$.
\end{prop}

 The proof below provides an identification of the limiting points before the explosion time.
  Assumption {\bf (B2)} on the domination of the mating function could also be relaxed before explosion using localization argument.

 \bigskip Let us apply the approach developed in \cite{BCM} for the asexual case. The method is based on the convergence of the characteristics of the associated semi-martingales developed in  Jacod-Shiryaev \cite{JS},  with the use of a specific functional space. This latter  exploits the population recurrence-type structure and the independence of the random variables $\{ M_0^N,F_0^N,\mathcal E_{n,p}^{f,N},   \mathcal E^{m,N}_{n,p}, (L_{n,p}^{f,N},L^{m,N}_{n,p}), n,p \geq 1\} $. This method allows us to prove tightness and identification under the optimal moment assumption on the jump measure, see Assumption {\bf A}. 
 
 \bi  Let us quickly summarize what we will do. We first remark that depending on the reproduction laws, we can have explosion of the process under Assumption {\bf A}. To deal with  this problem and to guarantee the boundedness of the characteristics, we  compactify the process as in \cite{BCM} by considering the new process $X^N$ defined as follows :
$$X^N_n=\left(\exp(-F^N_n/N),\exp(-M^N_n/N)\right).$$
This exponential transform combined with a functional space $\mathcal H$ formed by polynomials allow to exploit independence and positivity of the reproduction random variables.

\me  (I) In our setting, the characteristics of the exponential transform of the process  are given by
 formulas \eqref{characteristics} and \eqref{GN} below. 
 It has been proved in \cite{BCM} (see also Appendix A) that  their uniform convergence, in the sense of Lemma \ref{Gx} below guarantees the tightness of the sequence $(X^N_{[v_{N}.]})_{N}$ and yields the  characteristics of limiting semimartingales. 
 
 \me (II) To identify the limiting values as solutions of a stochastic differential equation, we need to   exploit   the explicit form given in Lemma \ref{Gx}. This representation
 is obtained  in Lemma \ref{LV}.
 
 \me (III) We come back to the initial process $Z^N$ using It\^o's formula, up to the explosion time and prove that the limiting values of the sequence $(Z^N)_{N}$ are solutions of the stochastic differential system \eqref{EDS-limit}. This will complete the proof
 of Proposition \ref{prop=tightness}.\\

 \bi Let us now develop this program. 
 
 \me (I)   The first part consists in introducing   functional space $\mathcal H$ and in proving Assumption (H1)  recalled in Appendix A. This  assumption  ensures the convergence of the characteristics of the rescaled Markov chain $(X^N_{[v_{N}.]})_{N}$ for test functions belonging to  $\mathcal H$  and provides their limiting form. Note that since $(X^N_{[v_{N}.]})_{N}$ is bounded, Assumption (H0) of \cite{BCM}  is obvious.  \\
 We  consider the space $\mathcal U=[-1,1]^2$ and  the space of monomial functions (on $\mathcal U$) defined  by
$$\mathcal H=\left\{ (u_{1},u_{2}) \in \mathcal U \rightarrow  H_{i,j}(u_{1},u_{2})  =(u_{1})^i(u_{2})^j \ ;\  i\geq 0, j\geq 0, i, j\ne 0\right\}.$$
Following \cite{JS,BCM}, 
  we consider the following  family of linear operators characterizing the law of the increments of the scaled Markov chain. It is defined  for $H$ measurable and bounded and for  $x=(\exp(-y),\exp(-z))\in {\cal X}=(0,1]^2$ by
$${\cal G}^N_x(H)=v_N\E\left( H(X^N_1 - x )| X^N_0=x\right),$$ where
\ben
X^N_1-x&=&\displaystyle{\Bigg(e^{-y}\bigg(\exp\Big({-\frac{1}{N}\sum_{p=1}^{[Ny]} \mathcal E_{p}^{f,N}  -\frac{1}{N} \sum_{p=1}^{g_N([Ny],[Nz])} L_{p}^{f,N}}\Big) -1\bigg)},\\
&&\qquad  \qquad \qquad
 \displaystyle{e^{-z}\bigg(\exp\Big({-\frac{1}{N}\sum_{p=1}^{[Nz]} \mathcal E_{p}^{m,N}-\frac{1}{N} \sum_{p=1}^{g_N([Ny],[Nz])} L_{p}^{m,N}}\Big)-1\bigg)\Bigg)}
 \een

 \me   Assumption (H1.1,2) is a direct consequence of  Stone-Weierstrass theorem and  
 the convergence needed in (H1.3) is proved in forthcoming  Lemmas \ref{CV-charac} and \ref{Gx}. 
For that purpose, we set
\be
\label{characteristics}
A_{k,\ell}^N(x)&=&\E\left(\exp\left(-\frac{k}{N}\sum_{p=1}^{[Ny]} \mathcal E_{p}^{f,N}-\frac{\ell}{N}\sum_{p=1}^{[Nz]} \mathcal E_{p}^{m,N}-\frac{1}{N} \sum_{p=1}^{g_N([Ny],[Nz]) }L_{p}^{k,\ell,N}\right)\right),
\ee
where
$L_{p}^{k,\ell,N}=k L_{p}^{f,N}+\ell L_{p}^{m,N}$ and  using that $\sum_{k=0}^{i}\sum_{\ell=0}^j  (-1)^{i-k+j-\ell}\binom{i}{k}\binom{j}{\ell} =0$, we get by expansion 
\be
\label{GN}
{\cal G}_x^N(H_{i,j})&=&e^{-iy-jz}\sum_{k=0}^{i}\sum_{\ell=0}^j  (-1)^{i-k+j-\ell}\binom{i}{k}\binom{j}{\ell} v_N \big(A_{k,\ell}^N(x)-1\big).
\ee
Furthermore, we set for $u,v \in \R$,
$$ f_{k}(u) = 1 - e^{-ku}, \qquad f_{k,\ell}(u,v)= 1 - e^{- ku-\ell v}$$
for  $k,\ell  \in \mathbb{N}$, and 
by independence of the reproduction events, we have
\begin{equation}
\label{techno}
A_{k,\ell}^N=a_1^Na_2^Na_3^N,
\end{equation}
for any $x=(\exp(-y),\exp(-z))\in (0,1]^2$, where
\begin{eqnarray*}
a_1^N(x)&=&\exp\left([Ny] \log\left(1-\epsilon^N_1 \right)\right), \qquad  \quad \epsilon^N_1 =\E\left(f_{k}\left({\mathcal E}^{f,N}/N\right)\right)\\
a_2^N(x)&=&\exp\left({[Nz] \log\left(1- \epsilon^N_2 \right)}\right), \qquad  \quad \epsilon^N_2 = \E\left(f_{\ell}\left({\mathcal E}^{m,N}/N\right)\right)\\
a_3^N(x)&=& \exp\left(g_N([Ny],[Nz]) \log\left(1-\epsilon^N_3 \right)\right), \quad  \epsilon^N_3 =\E\left(f_{k,\ell}\left((L^{f,N},L^{m,N})/N\right)\right).
\end{eqnarray*} 

\bigskip 
We use the following functions $f_{k}$ and $f_{k,\ell}$ and   their decompositions 
\begin{eqnarray*}
f_k(u)&= &kh(u)-{k^2\over 2} h^2(u)+R_k(u),\\
 f_{k,\ell}(u,v)&= &k h(u) +\ell h(v) - {k^2\over 2} h^2(u) -{\ell^2\over 2} h^2(v) - k\ell h(u)h(v) +R_{k,\ell}(u,v),
 \end{eqnarray*} 
 where  $R_{k}$ (resp.  $R_{k,\ell}$) is continuous  bounded and  $o(u^2)$ (resp. $o(\|(u,v)\|^2)$) in a neighborhood of $0$  (resp.
  $(0,0)$). These decompositions allow us to derive the asymptotic behavior 
  of $\epsilon_{i}^N$ from  Assumption {\bf A}, by summing the three components.
  Indeed, $R_{k}$ (resp.  $R_{k,\ell}$) are not null in a neighborhood of zero but small enough and  a simple approximation argument, which follows e.g. \cite[Section 4]{BCM}, yields
  $v_NN\E(R_k(L^{\bullet,N}))\rightarrow\int R_{k} d\nu_{\bullet}$ and
  $v_NN\E(R_{k,\ell}((L^{f,N},L^{m,N})/N))\rightarrow \int R_{k,\ell} d\nu_S$ as $N\rightarrow \infty$.
We get
 \be
 v_NN\,\epsilon^N_1\, \stackrel{N\rightarrow\infty}{\longrightarrow}\, \gamma_{k}^f&=&\alpha_{f}\, k - \frac{1}{2}\,\sigma_{f}^2 \, k^2 +  \int_{0}^{\infty} \big(f_k(u) - k h(u)\big) \nu_{f}(du) , \label{conve1}\\
  v_NN\,\epsilon^N_2\, \stackrel{N\rightarrow\infty}{\longrightarrow}\,\gamma_{\ell}^m&=&\alpha_{m}\, \ell - \frac{1}{2}\,\sigma_{m}^2 \,\ell^2 +  \int_{0}^{\infty} \big(f_{\ell}(u) - \ell h(u)\big) \nu_{m}(du), \label{conve2}\\
 v_NN \, \epsilon^N_3  \, \stackrel{N\rightarrow\infty}{\longrightarrow}\, \gamma_{k,\ell}^S&=&\alpha_{f}^S\, k+\alpha_m^S \, \ell- \frac{1}{2}\,(\sigma_{f}^S)^2 \,k^2- \frac{1}{2}\,(\sigma_{m}^S)^2 \, \ell^2 -(\sigma_{fm}^S)^2k\ell \nonumber \\
&&\quad +  \int_{\R_{+}^2} \big(f_{k,\ell}(u_1, u_2) - k h(u_1)-\ell h(u_2)\big) \nu_{S}(du_1du_2), \label{conve3}
\ee
where  we recall that $\sigma_{\bullet \bullet}^S$ is denoted by $\sigma_{\bullet}^S$.  

\bi
Letting $N\rightarrow \infty$, we obtain the following uniform convergence: 
\begin{lem}
\label{CV-charac}
For any $(i,j) \in \mathbb{N}^2\setminus \{(0,0)\}$, for any $k, \ell \in \mathbb{N}^2$,
\ben
\sup_{x\in (0,1]^2} e^{-iy-jz}\, \Big| v_N(A_{k,\ell}^N(x)-1)  + \gamma_k^fy+\gamma_\ell^mz+\gamma_{k,\ell}^Sg(y,z) \Big| \stackrel{N\to \infty}{\longrightarrow} 0,
\een
where $x=(e^{-y},e^{-z})$.
\end{lem}

 \me \begin{proof} We use the expression \eqref{techno} which is rewritten :
\be
\label{Aussois} A_{k,\ell}^N &=& 1 + (a_{1}^N-1)+(a_{2}^N-1)+(a_{3}^N-1)+(a_{1}^N-1)(a_{2}^N-1) +(a_{1}^N-1)(a_{3}^N-1) \nonumber \\
 &&+(a_{2}^N-1)(a_{3}^N-1)+(a_{1}^N-1)(a_{2}^N-1)(a_{3}^N-1)\ee
 for a convenient Taylor expansion.
We show now the uniform convergences 
 \be
 \label{cvunif1}
 \sup_{x\in (0,1]^2} e^{-(iy+jz)/3}\left\vert v_{N}(a_{p}^N(x) -1)- \gamma_{p}(x)    \right\vert \stackrel{N\to \infty}{\longrightarrow} 0,
 \ee
 for  $p=1,2, 3,$  where
 $$\gamma_{1}(x)= \gamma^f_{k} \,y \ ;\ \gamma_{2}(x)=  \gamma^m_{\ell}\, z\ ; \ \gamma_{3}(x) = \gamma_{k,\ell}^S \,g(y,z).$$

The  terms for $p=1,2$ correspond to the scaling of a Galton-Watson process and have already been considered in \cite{BCM}. 
Hence, we  focus on the third term, which is more delicate.   Using \eqref{conve3} and  Assumption {\bf (B.1)}, we first expand  
$$a_{3}^N(x)=e^{g_N([Ny],[Nz]) \log\left(1-\epsilon^N_3 \right)}=1+ {1\over v_{N}}\left(\gamma_{k,\ell}^S\,g(y,z)  + (g(y,z)+1) o_{N}(1)\right) + O\left({g(y,z)^2+1 \over v_{N}^2}\right),$$
as $N\rightarrow \infty$, uniformly for $x$ such that  $g_{N}(Ny,Nz)/N\leq v_{N}$. 
Combining this estimate and Assumption {\bf (B.2)} yields  $$ \sup_{ g_{N}(Ny,Nz)/N \le v_{N}} e^{-(iy+jz)/3} \left\vert v_{N}(a_{3}^N(x)
-1) - \gamma_{k,\ell}^S\,g(y,z) \right\vert \stackrel{N\to \infty}{\longrightarrow} 0.$$

\me Besides,  $\gamma^* = \sup_{N} \Big\{N v_{N} \log(1-\epsilon_{3}^N)\Big\}$ is finite 
since $Nv_N\epsilon_{3}^N$ has a finite limit. 
For $x$ such that  $g_{N}(Ny,Nz)/N \geq v_{N}$, we have
\ben
e^{-(iy+jz)/3}    v_N(a_{3}^N(x)+1)& \leq & e^{-(iy+jz)/3}
 \frac{g_{N}(Ny,Nz)}{N} (e^{(\gamma^*/v_{N}) g_{N}(Ny,Nz)/N}+1)\\
&\leq & e^{-(iy+jz)/3} (g(y,z)+o(1))(e^{(\gamma^*/v_{N}) (g(y,z)+o(1))}+1),
\een
where $o(1)$ is uniform with respect to $x$ using \eqref{CVmating}.
Recalling \eqref{g-exp},
we get that both $e^{-(iy+jz)/3} v_N(a_{3}^N-1)$ and $e^{-(iy+jz)/3}\gamma_3(x)$
converge to $0$ as $N$ tends to infinity, uniformly for $g_{N}(Ny,Nz)/N \geq v_{N}$.
This ends the proof of \eqref{cvunif1}.

\me Combining the three uniform convergences in \eqref{Aussois} yields the conclusion. 
\end{proof}

\bi  We can now compute the limit of \eqref{GN}, as $N$ tends to infinity, which is achieved in the following lemma. 

\me \begin{lem} 
\label{Gx} For any $(i,j) \in \mathbb{N}^2\setminus \{(0,0)\}$, we have 
 $$
 \sup_{x\in (0,1]^2} \left\vert {\cal G}_x^N(H_{i,j})-{\cal G}_x(H_{i,j}) 
 \right\vert \stackrel{N\to \infty}{\longrightarrow}  0,
 $$ 
 where, writing $x=(e^{-y},e^{-z})$ and  denoting by $\delta_{i,j}$  the Kronecker symbol:  
\ben
 -e^{iy+jz}{\cal G}_x(H_{i,j})&=&y\,\delta_{j,0}\left(\delta_{i,1}\alpha_f-(2\delta_{i,2}+\delta_{i,1})\sigma_f^2/2 +\int_0^{\infty}  \left( (-1)^{i+1}f_1(u)^i-\delta_{i,1}h(u)\right)\nu_f(du)\right)\\
 &&+z\,\delta_{i,0}\left(\delta_{j,1}\alpha_m- (2\delta_{j,2}+\delta_{j,1})\sigma_m^2/2 +\int_0^{\infty}  \left( (-1)^{j+1}f_1(u)^j-\delta_{j,1}h(u)\right)\nu_m(du)\right)\\
 &&+g(y,z)\bigg(\delta_{j,0}\left[\delta_{i,1}\alpha_f^S-(2\delta_{i,2}+\delta_{i,1})(\sigma_{f}^S)^2/2\right] +\delta_{i,0}\left[\delta_{j,1}\alpha_m^S-(2\delta_{j,2}+\delta_{j,1})(\sigma_{m}^S)^2/2\right]\\
 &&\qquad \qquad +\delta_{i,1} \delta_{j,1}(\sigma_{fm}^S)^2+\int_{[0,\infty)^2} g_{i,j}(u_{1}, u_{2})\nu_{S}(du_{1}, du_{2})\bigg),
 \een
  and
 \ben 
 g_{i,j}(u)&=&\delta_{j,0}\big((-1)^{i+1}f_1(u_1)^i-\delta_{i,1}h(u_1)\big)\\
 &&+\delta_{i,0}\big((-1)^{j+1}f_1(u_2)^j-\delta_{j,1}h(u_2)\big)-(-1)^{i+j}f_1(u_1)^if_2(u_2)^j 1_{i\ne 0} 1_{j\ne 0}.
 \een
 \end{lem} 

\begin{proof}
Combining \eqref{GN}  and the uniform convergence of the previous lemma, we obtain that ${\cal G}_x^N(H_{i,j})$ converges uniformly, as $n$ tends to infinity, to ${\cal G}_x(H_{i,j})$ which satisfies
$$-e^{iy+j z}{\cal G}_x(H_{i,j})=\sum_{k=0}^{i}\sum_{\ell=0}^j  (-1)^{i-k+j-\ell}\binom{i}{k}\binom{j}{\ell} (y \gamma_k^f+z\gamma_\ell^m+g(y,z)\gamma_{k,\ell}^S).$$
Plugging  the expressions of the constants $\gamma$ given in
\eqref{conve1} and \eqref{conve2} and \eqref{conve3}, the sum  above can be simplified 
using  $f_{k,\ell}(u_1,u_2)=f_{k}(u_1)+f_{\ell}(u_2)-f_{k}(u_1)f_{\ell}(u_2)$ and 
$$\  \sum_{k=0}^{i} \binom{i}{k}(-1)^{i-k}  = \delta_{0,i}
\quad ;\quad   \sum_{k=0}^{i} \binom{i}{k}(-1)^{i-k} k = \delta_{1,i},$$
$$\ \sum_{k=0}^{i} \binom{i}{k}(-1)^{i-k} \,k^2 = 2 \delta_{2,i} + \delta_{1,i}\quad ;\quad 
\ \sum_{k=0}^{i} \binom{i}{k}(-1)^{i-k} f_k(u) =  (-1)^{i+1}f_1(u)^i1_{i>0}.$$
We obtain the expected result.
\end{proof}

\bi 
(II) We now proceed  with the representation of the limiting points.
For that purpose, we proceed with the successive identification of the coefficients of the stochastic differential equation associated with the limiting characteristics obtained above.
Firstly, we gather the jump terms in a common Poisson representation. Indeed, considering first  ${\cal G}_x(H_{i,j})$ for $i+j\geq 3$ leads us to 
define the measure $\mu$ on $V=\{1,2,3\}\times [0,+\infty)\times[0,\infty)^2$ by
\be
\label{mu}\mu(dk, d\theta, du_{1},du_{2})& =& \delta_1(dk)\, d\theta \, \nu_{f}(du_{1}) \delta_0(du_{2}) + \delta_2(dk)\, d\theta \,  \delta_0(du_{1}) \, \nu_{m}(du_{2})\nonumber\\
&&\qquad \qquad  \qquad \qquad   \qquad \qquad  + \delta_3(dk) \, d\theta \, \nu_{S}(du_{1},du_{2}),
\ee
where $\delta_{k}$ is the Dirac mass in $k$. 
The jump image function $K=(K_1,K_2)$ is  the measurable function $K: (x,v)\in [0,1]^2\times V\rightarrow K(x,v)\in \R^2$  given by 
\be
\label{K1} K_{1}(x,v)=K_{1}(x,k,\theta, u_{1},u_{2}) =  - e^{-y}. \bigg(  f_{1}(u_{1})\, \un_{k=1, \, \theta \leq y} + f_{1}(u_{1})\, \un_{k=3, \, \theta \leq g(y,z)}\bigg), \\
\label{K2} K_2(x,v)= K_{2}(x,k,\theta, u_{1},u_{2}) =  - e^{-z}. \bigg(  f_{1}(u_{2})\, \un_{k=2, \, \theta \leq z} + f_{1}(u_{2})\, \un_{k=3, \, \theta \leq g(y,z)}\bigg),
\ee
where we recall that $x=(\exp(-y), \exp(-z))$.
Let us observe that $ \int_{V}   1\wedge\vert K(.,v)\vert^2  \mu(dv)<+\infty.$ 
Secondly, using ${\cal G}_x(H_{i,j})$ for $i+j=2$, we define  the diffusion coefficients
$\sigma(.)\in \mathcal M_{2,4}(\R)$  as follows
$$\sigma_{11}(x)= e^{-y} \sqrt{y}\sigma_f,\quad  \sigma_{12}(x)=0, \quad  \sigma_{21}(x)=0,  \quad \sigma_{22}(x)= e^{-z} \sqrt{z}\sigma_m,$$
and
$$\sigma_{13}(x)= e^{-y} \sqrt{g(y,z)}\sqrt{(\sigma^S_f)^2-  (\sigma^S_{fm})^4/(\sigma^S_m)^2}, \quad \sigma_{14}(x)= e^{-y} \sqrt{g(y,z)}(\sigma^S_{fm})^2/\sigma^S_m$$
$$ \sigma_{23}(x)=0,  \quad \sigma_{24}(x)= e^{-z} \sqrt{g(y,z)}\sigma^S_m.$$
Finally we set the drift term  $b(.)=(b_1(.),b_2(.))\in \R^2$:
\ben
b_1(x)={\cal G}_x(H_{1,0})&=&e^{-y}y\left( -\alpha_f+\frac{\sigma_f^2}{2}-\int_0^{\infty}  \left( f_1(u)-h(u)\right)\nu_f(du)\right) \\
&&+e^{-y}g(y,z)\left(-\alpha_f^S+\frac{(\sigma_{f}^S)^2}{2}-\int_0^{\infty}  \left( f_1(u_1)-h(u_1)\right)\nu_S(du_1,du_2) \right); \\
b_2(x)={\cal G}_x(H_{0,1})&=&e^{-z}z\left( -\alpha_m+\frac{\sigma_m^2}{2}-\int_0^{\infty}  \left( f_2(u)-h(u)\right)\nu_m(du)\right) \\
&&+e^{-z}g(y,z)\left(-\alpha_f^S+\frac{(\sigma_{f}^S)^2}{2}-\int_0^{\infty}  \left( f_1(u_2)-h(u_2)\right)\nu_S(du_1,du_2) \right).
\een
These parameters yield the following representation of  the limiting points of $(X^N_{[v_N.]})_N$.
\begin{lem}
\label{LV} Any limiting value in $\mathbb D([0,\infty), [0,1]^2)$ of the sequences of processes $\left(X^N_{[v_{N}.]}\right)_N$ is  a   semimartingale    solution of the  stochastic differential system
 \be
 \label{eds}
 X_{t}&=&X_{0} + \int_{0}^t b(X_{s}) ds + \int_{0}^t \sigma(X_{s}) dB_{s} + \int_{0}^t\int_{V} K(X_{s-},v) \tilde N(ds, dv),
 \ee
where  
 $B$ is a $4$-dimensional Brownian motion and
    $N$ is a Poisson point  measure on $\R_{+}\times V$ with intensity $ds \mu(dv)$,   $X_{0}, B$, $N$ are independent and $\tilde N$ is the   compensated martingale measure of $N$.
\end{lem}

\me
\begin{proof} We need  to prove  that  
(H2) in \cite{BCM} (cf. Appendix A) is satisfied.  The continuity of $x\in \mathcal X \rightarrow \mathcal G_x(H)$ for $H\in \mathcal H$ is a direct consequence of the continuity of $g$, which is guaranteed by $({\bf B3})$. The continuous extension to $\overline{\mathcal X}$ is due to \eqref{g-exp}.
 Using our definition of parameters $b,\sigma, K,\mu$, let us  now  check   that for   any $H\in \mathcal H$,
 \be
\label{idtfH}
 {\cal G}_{x}(H) =  \sum_{a\in\{1,2\}} \alpha_{a}(H) b_{a}(x)  + \sum_{a,b\in \{1,2\}}  \beta_{a,b}(H)c_{a,b}(x)   + \int_{V} \overline{H}(K(x,v)) \mu(dv), 
\ee
where  for  any $a,b\in \{1,2\}$,
$$c_{a,b}(x)=\sum_{i=1}^4\sigma_{a,i}(x)\sigma_{b,i}(x) +\int_{V} K_aK_b(x,v) \mu(dv)$$
and  $\alpha_{a}(H), \beta_{a,b}(H)$ are  the first and second order  coefficients of $H$ in its Taylor expansion and $\overline{H}=H- \sum_{a\in\{1,2\}} \alpha_{a}(H)  - \sum_{a,b\in \{1,2\}}  \beta_{a,b}(H)$ is the remaining term. 
We first observe that for   $H\in \mathcal H$, these   coefficients   are trivial.  There is a unique coefficient which is non zero for $H_{i,j}$ when $i+j\leq  2$ and it is equal to  $1$. Besides 
dor
  $i+j \geq 3$, $H_{i,j}=\overline{ H_{i,j}}$ and $\alpha_{.}(H_{i,j}) = \beta_{.,.}(H_{i,j})=0$.  
Then using  the triplet $(V,\mu,K)$ introduced above, we directly  check that   $$\mathcal G_{x}(H_{i,j})=\int_V H_{i,j}(K(x,v))\mu(dv)=\int_V \overline{H}_{i,j}(K(x,v))\mu(dv)$$ and $H_{i,j}$ satisfies   \eqref{idtfH} for $i+ j \geq 3$.
Then we can  check that \eqref{idtfH} is satisfied for  $H_{2,0}$  :
\ben {\cal G}_x(H_{2,0})&=e^{-2y}y\left( \sigma_f^2+\int_0^{\infty} f_1(u)^2\nu_f(du)\right)+e^{-2y}g(y,z)\left((\sigma_f^S)^2+\int_{[0,\infty)^2} f_1(u_1)^2 \nu_S(du_{1},du_{2}) \right)\\
&= \sum_{i=1}^4 \sigma_{1i}^2(x)+\int_V K_1^2(x,v) \mu(dv) =c_{1,1}(x).
\een
 Similarly 
\ben {\cal G}_x(H_{0,2})&=e^{-2z}z\left( \sigma_m^2+\int_0^{\infty} f_1(u)^2\nu_m(du)\right)+e^{-2z}g(y,z)\left((\sigma_m^S)^2+\int_{[0,\infty)^2} f_1(u_2)^2\nu_S(du_{1},du_{2}) \right)\\
&= \sum_{i=1}^4\sigma_{2i}^2(x)+\int_V K_2^2(x,v) \mu(dv )
=c_{2,2}(x),
\een
 and   \eqref{idtfH} is satisfied for  $H_{0,2}$. Finally, 
the crossed term writes
\ben
{\cal G}_x(H_{1,1})&=&e^{y+z}g(y,z)\left((\sigma_{fm}^S)^2+\int_{[0,\infty)^2}  f_1(u_1)f_1(u_2)\nu_{S}(du_{1}, du_{2}) \right)\\
&=& \sum_{i=1}^4\sigma_{1i}(x)\sigma_{2i}(x)+\int_V K_1(x,v)K_2(x,v) \mu(dv)=c_{1, 2}(x)=c_{2,1}(x)
\een
 and \eqref{idtfH}  is proved for any $H\in \mathcal H$, recalling that the definition of $b$  guarantees the identity for 
 $i+j=1$. This proves that  (H2)  is satisfied 
and recalling that (H1) is already proved,  we can apply Theorem 2.4 in \cite{BCM},  see also 
 Theorem \ref {identification}  in Appendix . It ends the proof.   \end{proof}
  
  \bi
  
 \bi (III)   Let us now come back to the initial processes. 
 
 We write $V = V_{1}\cup V_{2}$, where $V_1=\{1,2,3\}\times [0,+\infty)\times(0,1]^2$ and $V_2=\{1,2,3\}\times [0,+\infty)\times(1,\infty)^2$,  to split small 
and large jumps.

\bi We have seen in Lemma \ref{LV} that 
\ben
X_{t}^1= \exp(-F_{t}) &=& X_{0}^1 + \int_{0}^t \overline{b_{1}}(X_{s}
) ds +\sum_{i=1}^4  \int_{0}^t \sigma_{1,i}(X_{s})dB^i_{s} \\
&&+ \int_{0}^t \int_{V_1}K_{1}(X_{s-},v)\tilde N(ds,dv)+ \int_{0}^t \int_{V_2}K_{1}(X_{s-},v) N(ds,dv),
\een
where
\be
\overline{b_1}(x)&=&e^{-y}y\left( -\alpha_f+\frac{\sigma_f^2}{2}-\int_{(0,1]}  \left( f_1(u)-h(u)\right)\nu_f(du)+\int_{(1,\infty)}  h(u)\nu_f(du)\right) \label{bbar}  \\
&&+e^{-y}g(y,z)\left(-\alpha_f^S+\frac{(\sigma_{f}^S)^2}{2}+\int_{(0,1]}  \left( h(u_1)-f_1(u_1)\right)\nu_S(du_1,du_2)+\int_{(1,\infty)}  h(u_1)\nu_S(du_1,du_2) \right). \nonumber  
\ee
Using It\^o's formula we get before the explosion time $\mathbb T_e$:
\ben
\log X_{t}^1= -F_{t} &=&-F_{0} +\int_{0}^t {1\over X_{s}^1} \overline{b_{1}}(X_{s}) ds - {1\over 2} \sum_{i=1}^4\ \int_{0}^t {\sigma_{1,i}^2(X_{s})\over (X_{s}^1)^2}ds +\sum_{i=1}^4  \int_{0}^t {\sigma_{1,i}(X_{s})\over X_{s}^1}dB^i_{s}\\
&&  + \int_{0}^t \int_{V_1} \Big\{\log\big(X_{s-}^1+K_{1}(X_{s-},v)\big) -\log(X_{s-}^1)  \Big\}\tilde N(ds,dv)\\
&& +  \int_{0}^t \int_{V_1} \Big\{\log\big(X_{s-}^1+K_{1}(X_{s-},v)\big) -\log(X_{s-}^1)  + K_{1}(X_{s-},v){1\over X_{s-}^1} \Big\}\mu(dv)ds.\\
&&+ \int_{0}^t \int_{V_2} \Big\{\log\big(X_{s-}^1+K_{1}(X_{s-},v)\big) -\log(X_{s-}^1)  \Big\} N(ds,dv).
\een
By definition of the coefficients introduced previously and by identification of the Brownian terms, we obtain
$${\sigma_{1,1}(X_{s})\over X_{s}^1}= \sqrt{F_{s}} \sigma_{f}\ ;\ \sigma_{1,2}(X_{s})=0 \ ;$$
$${\sigma_{1,3}(X_{s})\over X_{s}^1}= \sqrt{g(F_{s}, M_{s})}\sqrt{(\sigma^S_f)^2-  (\sigma^S_{fm})^4/(\sigma^S_m)^2}
 \ ; \ {\sigma_{1,4}(X_{s})\over X_{s}^1}= \sqrt{g(F_{s}, M_{s})}\frac{(\sigma^S_{fm})^2}{\sigma^S_m}.$$
We also recall from \eqref{K1}  that for any positive two-dimensional $x$, for $v=(k,\theta, u_{1},u_{2})$, 
$${K_1(x,v)\over x^1} = - \Big(f_{1}(u_{1}){\bf 1}_{k=1, \theta\le y} + f_{1}
(u_{1}){\bf 1}_{k=3, \theta\le g(y,z)} \Big).$$
Writing $N^f(ds, d\theta , du)=N(ds, \{1\}, d\theta, du, \{0\})$ and $N^S(ds,d\theta, du_1,du_2)=N(ds, \{3\}, d\theta, du_1, du_2)$, computation gives that \ben
&& \int_{V_1} \Big\{\log\big(X_{s-}^1+K_{1}(X_{s-},v)\big) -\log(X_{s-}^1)  \Big\}\tilde N(ds,dv)\\
&&\qquad \quad =\int_{[0,\infty)\times[0,1)} 1_{\theta\leq F_{s-}} u\tilde N^f(ds,du)+\int_{[0,\infty)\times[0,1)} 1_{\theta\leq g(F_{s-},M_{s-})} u_1\tilde N^S(ds,du_1,du_2).
\een
We obtain similarly the last jumps terms, without compensation.
Finally, the drift term of $F$ is given by the remaining terms.  Recall that $\mu$ is defined in \eqref{mu} and replacing $ \overline{b_{1}}(x)$  by its value given in \eqref{bbar}, it is equal to
\ben
&& -{1\over x^1} \overline{b_{1}}(x)  + {1\over 2} \sum_{i=1}^4 {\sigma_{1,i}^2(x)\over (x^1)^2}ds + \int_{(0,1]}  \left(h(u)- f_1(u)\right)\nu_f(du)+\int_{(1,\infty)}  h(u)\nu_f(du) \\
&&\qquad \quad +\int_{(0,1]}  \left( h(u_1)-f_1(u_1)\right)\nu_S(du_1,du_2)+\int_{(1,\infty)}  h(u_1)\nu_S(du_1,du_2) \, =\, \alpha_f y\, +\ \alpha_f^Sg(y,z).
\een
 This yields the expected equation for $F_t$. Following the same lines for $M_t=-\log(X_t^2)$ ends the proof of 
Proposition \ref{prop=tightness}.

\subsection{Uniqueness and convergence}
\label{sec:proofThe2}

\me Using the results of forthcoming Section \ref{sec:uniqueness}, we are able to prove the uniqueness needed for Theorem \ref{main} in a slightly more general framework.  Recall that the measure $$
\nu(du_1,du_2)=\nu_f(du_1)\delta_0(u_2)+\nu_m(du_2)\delta_0(u_1)+\nu_S(du_1,du_2)
$$
has been introduced in Assumption {\bf C}.

\bi 
\begin{prop} 
\label{the:main2}
 Let us assume that Hypotheses {\bf (B2)--(B4)}  are satisfied and
 $$\int_{[0,\infty)^2} (u^2_1+u^2_2)\wedge 1  \, \nu(du_1,du_2)<\infty$$
 Let us moreover assume that  that
 there exists $\varepsilon_0>0$ such that 
\be
\label{slight}
\liminf_{a\to 0} \,\,e^{\varepsilon_0 \left(\int_{A(a)} (u_1+u_2) \,\,  \nu(du_1,du_2)\right)}
\int_{B(a)}(u^2_1+u^2_2)\,  \nu(du_1,du_2)=0,
\ee
with $A(a)=\{(u_1,u_2):\, a<u_1\le 1, a<u_2\le 1\}, B(a)=\{(u_1,u_2):\, 0<u_1\le a, 0<u_2\le a\}$.

\me  Then, the stochastic differential system \eqref{system:main} has a unique strong (positive) solution up to the explosion time 
$$
\T_{\mathbf{e}}=\lim_{n\to \infty} \inf\{t\ge 0: F_t\ge n \hbox{ or } M_t\ge n\}.
$$
If $\nu$ satisfies the extra assumption $\int_{[0,\infty)^2} (u^2_1+u^2_2)\wedge (u_1+u_2) \, \nu(du_1,du_2)<\infty$,
then $\TT_{\mathbf{e}}=\infty$ a.s.
\end{prop}

\me Note that  Assumption {\bf C} obviously  implies \eqref{slight}. Observe also  that under Assumptions {\bf (B2)--(B4)}, $g$ is locally Lipschitz with linear growth and satisfies the ellipticity assumption.

\me The proof of Proposition \ref{the:main2} is a simple adaptation
of the proof of uniqueness of the next section. The measure that plays the role of $\lambda$ in Section  \ref{sec:uniqueness}, is $\nu$.
The representation of jumps  in  \eqref{system:main} relies on the three Poisson point measures $N^f, N^m, N^S$. These measures  can be gathered in a single Poisson point measure for convenience.


\me Finally, combining Propositions \ref{prop=tightness} and \ref{the:main2},  we have proved the convergence stated in Theorem \ref{main}.

\section{Pathwise uniqueness}
\label{sec:uniqueness}

We have seen previously that the main technical problems to prove uniqueness for the system \eqref{system:main}, come from 
the presence of the square root
as coefficient on the Brownian terms, the presence of singular coefficients for the compensated 
Poisson terms and the fact that this is 
a two-dimensional system. In this section,  we present a simpler version of this system by focusing on the sexual coupling term.  This system contains all the
difficulties mentioned, improving the known results in the literature. We do this to keep notation as simple as possible.  Without additional complexity, we actually consider here a more general diffusion and jump terms.

\subsection{The system of equations}
\label{subsec:general}

\me We study the uniqueness problem  for the following system of stochastic differential equations. This system has a form similar to the one obtained in \eqref{system:main} and contains all its difficulties. It is given by 
 
\begin{eqnarray}
\label{eq:e1}
&X_t&=x_0 +\int_0^t b_1(X_s,Y_s) ds + \int_0^t \sqrt{\ell_1(X_s,Y_s)} \, dB^1_s\nonumber\\&&+ 
\int_0^t \int_{\RR_{+}^2}  \ind_{\theta\le \kappa_1(X_{s-},Y_{s-})} p_1(X_{s-},Y_{s-}) h(z)  \,  \nt^1(ds, d\theta,dz)\nonumber\\
&&+\int_0^t \int_{\RR_{+}^2}  \ind_{\theta\le \kappa_1(X_{s-},Y_{s-})} p_1(X_{s-},Y_{s-}) (z -h(z)) \, N^1(ds, d\theta,dz);\nonumber\\
&Y_t&=y_0 +\int_0^t b_2(X_s,Y_s) ds + \int_0^t \sqrt{\ell_2(X_s,Y_s)} \, dB^2_s \nonumber\\ &&+ 
\int_0^t \int_{\RR_{+}^2} \ind_{\theta\le \kappa_2(X_{s-},Y_{s-})} p_2(X_{s-},Y_{s-}) h(z) \,  \nt^2(ds, d\theta,dz)\nonumber\\
&&+\int_0^t \int_{\RR_{+}^2}  \ind_{\theta\le \kappa_2(X_{s-},Y_{s-})} p_2(X_{s-},Y_{s-}) (z -h(z)) \, N^2(ds, d\theta,dz).
\end{eqnarray}
The processes  $B^1$ and $B^2$ are Brownian motions  and $N^1$ and $N^2$ are Poisson point  measures on 
$(\mathbb{R}_{+ })^3$  with intensities 
$ds\, d\theta\, \lambda_{1}( dz)$ and $ds\, d\theta\, \lambda_{2}( dz)$, not necessarily independent. 

\me In what follows, we will denote
$$
\lambda(dz) = \lambda_{1}(dz) + \lambda_{2}(dz),
$$
and throughout this section we assume that $\lambda$ satisfies the hypothesis 
\begin{enumerate}[{\bf (F0)}]
\item $\int_0^\infty (z^2 \wedge 1) \, \lambda(dz)<\infty.$
\end{enumerate}

\bi The coefficient are defined   on $\RR_{+}^2$ 
and for $i=1,2$ the hypotheses about these  coefficients are
\begin{enumerate}[{\bf (F1)}]
\item $b_i,\ell_i,\kappa_i,p_i$ are locally Lipschitz on  $\RR_{+}\times\RR_{+}$. We also assume
that for all $z\in \RR_+$ it holds $b_i(0,z)=\ell_i(0,z)=\kappa_i(0,z)=b_i(z,0)=\ell_i(z,0)=\kappa_i(z,0)=0$.

\item $\ell_i,\kappa_i,p_i$ are nonnegative, and $p_i$ are strictly positive in every compact set  of $ [0,\infty)^2$.

\item $b_i,\ell_i,\kappa_i$ have linear growth and $p,q$ are bounded. We denote
by $\mathbf{L}, \mathbf{A}$ two constants such that
$$
|b_1(x,y)|+|b_2(x,y)|+\ell_1(x,y)+\ell_2(x,y)+\kappa_1(x,y)+\kappa_2(x,y)\le \mathbf{L}(x+y)+ \mathbf{A}.
$$
We assume without loss of generality that $p_i$ are bounded by 1.

\item The function $h\in C_{b}( \RR_{+}, \RR_{+})$ and it satisfies $h(z)=z$ in a neighborhood of $0$.
\end{enumerate}

\bi We point out the following facts that are direct consequences of {\bf (F0)} and {\bf (F4)}.
\begin{enumerate}
\item $\int_0^1 z^2 \, \lambda(dz)<\infty$, $\int_0^1 h^2(z) \, \lambda(dz)<\infty$, $\int_0^1 |z-h(z)| \, \lambda(dz)<\infty$
and $\lambda([1,\infty))<\infty$.
\item  $\int_1^\infty z \, \lambda(dz)<\infty$ if and only if $\int_0^\infty |z-h(z)| \, \lambda(dz)<\infty$.
\item $\int_0^\infty z^2\wedge z \, \lambda(dz)<\infty$ if and only if $\int_0^\infty h^2(z)\, \lambda(dz)<\infty$ and
$\int_0^\infty |z-h(z)|\, \lambda(dz)<\infty$.
\end{enumerate}

\me Note also that because of {\bf (F1)},  $(0,0)$ is an absorbing point and any solution issued from $\mathbb{R}_{+}^2$ stays in $\mathbb{R}_{+}^2$. 

\bi In some of the computations below, we shall use Burkh\"older-Davis-Gundy inequality with $p=1$, which provides
a finite constant $\mathbf{C}_1$, such that \be \label{BDG}\EE\left(\sup\limits_{s\le \tau} |M_s| \right)
 \le \mathbf{C}_1 \, \EE([M,M]_{\tau}^{1/2})\ee for any local martingale $M$, and any stopping
 time $\tau$ (cf. Dellacherie-Meyer \cite{DM} VII.92). We also need to use similar inequalities relating the supremum of a local martingale
 and its predictable quadratic variation. Namely, there exists a constant 
 $\overline{\mathbf{c}}_1>0$, such that if the jumps of $M$ are bounded 
in absolute value by $\pmb\Delta$ then (see Lenglart-L\'epingle-Pratelli \cite{Lenglart})
\begin{eqnarray}
\overline{\mathbf{c}}_1\EE(\langle M,M\rangle_{\tau}^{1/2})&\le& \EE\left(\sup\limits_{s\le \tau} |M_s| \right)+\pmb\Delta;\nonumber\\
\EE([M,M]_{\tau}^{1/2})&\le& 3\, \EE(\langle M,M\rangle_{\tau}^{1/2}).\label{ineq:BDG-Lenglart}
\end{eqnarray}
Note that if $M_{t} = \int_{0}^t \int_0^\infty H_{s-}(z) \tilde N(ds, dz)$ where $N$ is a Poisson point measure with intensity $\nu$, then $[M,M]_{\tau}= \int_{0}^t \int_0^\infty H^2_{s-}(z)  N(ds, dz)$ and 
$\langle M,M\rangle_{\tau} = \int_{0}^{\tau} \int_0^\infty H^2_{s-}(z)  \nu(ds, dz)$.

\bi Our first result is an a priori bound for system \eqref{eq:e1} and  we
set
$$X^*_{t}= \sup_{s\le t} |X_{s}|.$$

\begin{prop} 
\label{pro:2} Assume that $(x_{0}, y_{0}) \in \mathbb{R}_{+}^2$.
Assume that $\int_0^\infty (z^2\wedge z) \, \lambda(dz)<\infty$ and {\bf (F1)}--{\bf (F4)} hold. 
If $(X,Y)$ is a nonnegative solution of 
\eqref{eq:e1} then, the following a priori estimates hold for all $t>0$
$$
\EE(X_t+Y_t)\le (x_0+y_0+a\mathbf{A}t)e^{a \mathbf{L}\, t}
$$
and
$$
\EE(X^*_t+Y^*_t)\le \big(x_0+y_0+D+(D+a)\mathbf{A}\, t\big) e^{(D+a)\mathbf{L}\, t},
$$
where $\mathbf{L}$ and $\mathbf{A}$ are given in {\bf (F3)},  $a= 2+\int_0^\infty |z-h(z)|\, \lambda(dz)$ and 

$D=\mathbf{C}_1\left(2+\sqrt{\int_0^\infty  h^2(z) \, \lambda_1(dz)}+\sqrt{\int_0^\infty h^2(z) \, \lambda_2(dz)}\right)$.
\end{prop}

\smallskip
\begin{proof} We consider $S^X_n=\inf\{t>0: X_t\ge n\}, S^Y_n=\inf\{t>0: Y_t\ge n\}$
and $S_n=S^X_n\wedge S^Y_n$.
Then, we have
\begin{eqnarray}
\EE(X_{t\wedge S_n})&=&x_0+\int_0^t \EE(b_1(X_{s},Y_{s}), s<S_n) \, ds \label{eq:e2}\\
&&+\int_0^\infty (z-h(z)) \, \lambda_1(dz)\, \int_0^t \EE(\kappa_1(X_{s},Y_{s}) p_1(X_{s},Y_{s}), s<S_n ) \, ds \nonumber \\
&\le& x_0 +\big(1+ \int_0^\infty |z-h(z)| \, \lambda_1(dz)\big)\mathbf{A}t+\\
&&\qquad \qquad
\big(1+ \int_0^\infty |z-h(z)|  \, \lambda_1(dz)\big)\mathbf{L} \int_0^t \EE(X_{s\wedge S_n}+
Y_{s\wedge S_n}) \, ds. \nonumber
\end{eqnarray}
Proceeding similarly for $Y$, this implies that
$$
\begin{array}{l}
\EE(X_{t\wedge S_n}+Y_{t\wedge S_n})\le x_0+y_0 +a\mathbf{A} t+
a\mathbf{L}\int_0^t \EE(X_{s\wedge S_n}+ Y_{s\wedge S_n}) \, ds.
\end{array}
$$
To apply Gronwall's inequality, we need to bound $\EE(X_{t\wedge S_n}+Y_{t\wedge S_n})$. This is not direct because the processes may jump at $S_{n}$.  

The first lines of \eqref{eq:e2} show that 
$\EE(X_{t\wedge S_n})\le x_0+ t(2\mathbf{L} \,n+\mathbf{A})(1+\int_0^\infty |z-h(z)|  \, \lambda_1(dz))$, proving
that for all $t$ we have $\EE(X_{t\wedge S_n})<\infty$. A similar conclusion holds for $Y$. 

\medskip
From Gronwall's inequality, using   that $X,Y$ are nonnegative and they have only
upward jumps, we obtain
$$
n\, \PP(S_n<t)\le \EE(X_{t\wedge S_n}+Y_{t\wedge S_n})\le (x_0+y_0 +a\mathbf{A} t) e^{a\mathbf{L}t},
$$
proving that $S_n \to \infty$ a.s., as $n\to \infty$. Now, Fatou's lemma shows that
\be
\label{majesp}
\EE(X_{t}+Y_{t})\le \liminf\limits_{n\to \infty} \EE(X_{t\wedge S_n}+Y_{t\wedge S_n})
\le (x_0+y_0 +a\mathbf{A} t) e^{a\mathbf{L} t}
\ee
which proves the first part of the lemma.
Besides,
\begin{eqnarray*}
\EE(X^*_{t\wedge S_n})& \le &x_0+\int_0^t \EE(|b_1(X_{s},Y_{s})|, s<S_n) \, ds\\
&& +\int_0^\infty |z-h(z)|  \, \lambda_1(dz)
\int_0^t \EE(\kappa_1(X_{s},Y_{s}) p_1(X_{s},Y_{s}), s<S_n ) \, ds\\
&&+\EE\left(\sup\limits_{s\le t\wedge S_n} \left|\int_0^s 
\sqrt{\ell_1(X_s,Y_s)} \, dB^1_s\right|\right)\\
&&+\EE\left(\sup\limits_{s\le t\wedge S_n} \left|\int_0^s
\int_{[0,\infty)^2} \ind_{\theta\le \kappa_1(X_{s-},Y_{s-})}  p_1(X_{s-},Y_{s-}) h(z) \,  \nt^1(ds,d\theta,dz)\right|\right)
\end{eqnarray*}
Using  inequality \eqref{majesp} for the two first terms of the right hand side above and 
 \eqref{BDG} for the two last terms together with Cauchy Schwarz
for the jump term, we obtain
\begin{eqnarray*}
\EE(X^*_{t\wedge S_n})
&\le&  x_0+\big(1+\int_0^\infty |z-h(z)|  \, \lambda_1(dz)\big)
\left[\mathbf{A}t +\mathbf{L} \int_0^t \EE(X_{s\wedge S_n }+Y_{s\wedge S_n }) \, ds\right]\\
&&+\mathbf{C}_1\left(\EE\left(\int_0^{t\wedge S_n} \ell_1(X_s,Y_s)\, ds\right)\right)^{1/2}\\
&&+\mathbf{C}_1 \left(\EE\left(\int_0^{t\wedge S_n} \int_{[0,\infty)^2} \ind_{\theta\le \kappa_1(X_{s-},Y_{s-})}  p_1^2(X_{s-},Y_{s-}) h^2(z) \, N^1(ds,d\theta,dz)\right)\right)^{1/2},
\end{eqnarray*}
where we have used that the square root is a concave function. Therefore,
$$
\begin{array}{lll}
\EE(X^*_{t\wedge S_n})&\hspace{-0.3cm}\le x_0&\hspace{-0.3cm}+
\big(1+\int_0^\infty |z-h(z)|  \, \lambda_1(dz)\big)
\left[\mathbf{A}t +\mathbf{L} \int_0^t \EE(X_{s\wedge S_n }+Y_{s\wedge S_n }) \, ds\right]\\
\\
&&\hspace{-0.3cm}+\mathbf{C}_1\left(1+\sqrt{\int_0^\infty h^2(z) \, \lambda_1(dz)}\right)
\left(\mathbf{A}t+\mathbf{L} \int_0^t \EE(X_{s\wedge S_n }+Y_{s\wedge S_n })\, ds\right)^{1/2}\\
\\
&\hspace{-0.3cm}\le x_0&\hspace{-0.3cm}+d_1+d_2\mathbf{A}t+
d_2\, \mathbf{L}\int_0^t \EE(X_{s\wedge S_n }+Y_{s\wedge S_n }) \, ds,
\end{array}
$$
wiith $d_1=\mathbf{C}_1\left(1+\sqrt{\int_0^\infty h^2(z) \, \lambda_1(dz)}\right)$, $d_2=d_1+1+
\int_0^\infty |z-h(z)| \, \lambda_1(dz)$ (here we have used that $\sqrt{a}\le 1+a$). 
The result follows from Gronwall's inequality.
\end{proof}

\bi
\begin{prop} 
\label{prop:1} Let $(x_{0}, y_{0}) \in \mathbb{R}_{+}^2$.
We assume that $(X,Y)$ is a nonnegative solution of the system \eqref{eq:e1}, such that
$x_0=0$, and $\int_0^\infty (z^2\wedge z) \, \lambda(dz)<\infty$. 
If ${\bf (F1)}-{\bf (F4)}$ hold, then for all $t\ge 0$, we have $X_t=0, Y_t=y_0$. A similar conclusion holds if $y_0=0$.
\end{prop}
\begin{proof} Consider  $U_\epsilon=\inf\{t>0: X_t\ge \epsilon \hbox{ or } Y_t\ge y_0+\epsilon\}$ where $
\epsilon>0$.  Note that $U_\epsilon>0$ a.s. since the process $(X,Y)$ is right-continuous. 
As in the proof of the previous proposition and using $b_1(0,.)=\kappa_1(0,.)=0$, we have the following estimate 
\begin{eqnarray}
\EE(X_{t\wedge U_\epsilon})&\le& \int_0^t \EE(|b_1(X_{s},Y_{s})|, s<U_\epsilon) \, ds \nonumber \\
&&\qquad +\int_0^\infty |z-h(z)| \, \lambda_1(dz)\, \int_0^t \EE(\kappa_1(X_{s},Y_{s}) p(X_{s},Y_{s}), s<U_\epsilon ) \, ds \nonumber \\
&=&\int_0^t \EE(|b_1(X_{s},Y_{s})-b_1(0,Y_{s})|, s<U_\epsilon) \, ds \nonumber \\
&& \qquad +\int_0^\infty |z-h(z)| \, \lambda_1(dz)\, \int_0^t \EE(\kappa_1(X_{s},Y_{s}) p_1(X_{s},Y_{s})-
\kappa_1(0,Y_{s}) p_1(0,Y_{s}), s<U_\epsilon ) \, ds \nonumber \\
&\le &R\big(1+\int_0^\infty |z-h(z)| \, \lambda_1(dz)\big)\, \int_0^t \EE(X_{s},s<U_\epsilon) \, ds \nonumber \\
&\le &R\big(1+\int_0^\infty |z-h(z)| \, \lambda_1(dz)\big)\, \int_0^t \EE(X_{s\wedge U_\epsilon}) \, ds, \nonumber
\end{eqnarray}
where $R$ is a Lipschitz constant for $b_1, \kappa_1p_1$ on $[0,\epsilon]\times[0,y_0+\epsilon]$.
Gronwall's inequality gives that $\EE(X_{t\wedge U_\epsilon})=0$, which implies that
$X_{t\wedge U_\epsilon}=0$ a.s. and then $Y_{t\wedge U_\epsilon}=y_{0}$ a.s.. In particular, on the trajectories where  $U_{\epsilon}<\infty$, there is a small time $s_0>0$ such that for all $0\le s\le s_0$,
$X_{s+U_\epsilon}<\varepsilon$ and $Y_{s+U_\epsilon}<y_0+\epsilon$
(by right continuity), which gives a contradiction. Therefore, the only possible conclusion is that  $U_{\epsilon}=\infty$ and  we conclude that
$X_t=0$ and  $Y_t=y_0$, for all t.
\end{proof}

\subsection{Uniqueness}
In this section we shall prove pathwise uniqueness for the system \eqref{eq:e1}. 
We need the ellipticity assumption for the coefficients $\ell_i, i=1,2$, given in Assumption {\bf (B4)},
\begin{enumerate}[{\bf (F5)}]
\item For $i=1,2$ and for every $0<\delta\le n<\infty$, there exists $\zeta=\zeta(\delta,n)>0$ such that
$$
\zeta\le \inf\{\ell_i(x,y): \, (x,y)\in [\delta,n]^2\}.
$$
\end{enumerate}
We also need to have a control on the small jumps and this is done through the following hypothesis
on $\lambda$, which is the analog of \eqref{slight} in Proposition \ref{the:main2}.

\medskip
\begin{enumerate}[{\bf (F6)}]
\item There exists $\varepsilon_0>0$ such that 
\begin{equation}
\label{hyp:H}
\liminf\limits_{a\downarrow 0} \left[e^{\,\varepsilon_0\int_{a}^1 z\,  \lambda(dz)}\int_{0}^a z^2 \lambda(dz)\right]
=0.
\end{equation}
\end{enumerate}

\me We notice that if $\int_0^{\infty} (z\wedge 1)\, \lambda(dz)<+\infty$ then $\lambda$ satisfies hypotheses {\bf (F0)} and {\bf (F6)}. Also
it is quite direct to show that if $\mu\le \lambda$ and $\lambda$ satisfies  {\bf (F0)} and {\bf (F6)}, then
$\mu$ fulfils  {\bf (F0)} and {\bf (F6)}.

\me For a solution $(X,Y)$ of system \eqref{eq:e1}, we denote
by $\T_{\mathbf{e}}$ the explosion time of $(X,Y)$, which is given by 
$$
\T_{\mathbf{e}}=\lim\limits_{n\to\infty} S_n^X\wedge S_n^Y.
$$ 
Now, we are ready to state a uniqueness result. 

\begin{thm} 
\label{the:e1} Assume that $(x_{0}, y_{0}) \in \mathbb{R}_{+}^2$.
Assume that the coefficients of the system \eqref{eq:e1} satisfy {\bf (F1)--(F5)}, 
and $\lambda=\lambda_1+\lambda_2$ satisfies {\bf (F0)} and {\bf (F6)}. Then, pathwise uniqueness 
holds for this system, that is, if $(X,Y)$ and $(\wX,\wY)$ are two solutions
up to their respective explosion times $\T_{\mathbf{e}}$ and $\widetilde{\T}_{\mathbf{e}}$, 
then $\T_{\mathbf{e}}=\widetilde{\T}_{\mathbf{e}}$ a.s. and 
for all $t< \T_{\mathbf{e}}$ we have $(X_t,Y_t)=(\wX_t,\wY_t)$ a.s..

\me Under the extra hypothesis $\int_0^{\infty} (z^2 \wedge z)\, \lambda(dz)<\infty$, we have $\T_{\mathbf{e}}=\infty$ a.s.
\end{thm}

\begin{proof}  (i) In the first part of the proof, we assume the extra condition  
\begin{equation}
\label{eq:extra}
\int_0^{\infty} (z^2 \wedge z)\, \lambda(dz)<\infty
\end{equation}
and Proposition \ref{pro:2} guarantees non-explosion of solutions. Since $\lambda$ satisfies {\bf (F6)},
there exists $\varepsilon_0>0$ such that

$$
\liminf\limits_{a\downarrow 0}e^{\,\varepsilon_0 \int_a^\infty h(z)\,  \lambda(dz)}
\int_0^a h(z)^2 \lambda(dz)=0,
$$
because $h$ agrees with the identity on a neighborhood of $0$ and $h$ is bounded then $\int_1^\infty h(z) \lambda(dz)<\infty$. 
In what follows, we denote by $\NvDelta$ a bound for $h$. 

\medskip

We consider $(X,Y)$ and $(\wX,\wY)$ two strong solutions of the system
\eqref{eq:e1}.
Let us fix $0<\delta< x_0\wedge y_0$  and $n\in \mathbb{N}^*$ and let us take 
$T_\delta=\inf\{t>0: X_t\wedge Y_t \wedge \wX_t \wedge \wY_t < \delta\}$,
$S_n=S_n^X\wedge S_n^Y\wedge S_n^{\wX}\wedge S_n^{\wY}$, where we recall notation $S_n^Z=\inf\{t \geq 0 : Z_t\geq n\}$, and
\begin{equation}
\label{eq:Tnd}
T_{{n,\delta}}=T_\delta\wedge S_n.
\end{equation}

\medskip
We will prove that there exists  $t_{0}>0$ and a constant $A>0$ such that for all $t\le t_0$ 
\begin{equation}
\label{uni}
\EE((X-\wX)^*_\Td+(Y-\wY)^*_\Td)\le A \liminf_{a\to 0} \left(e^{\epsilon_0 \int_a^\infty h(z)\, \lambda(dz)}\int_0^a h(z)^2\, \lambda(dz)\,  
\right)^\frac12=0,
\end{equation}
where we recall that we write $Z^*_t=\sup\{ Z_s : s\leq t\}$.
Uniqueness will be shown on the interval $[0, t_0\wedge T_{{n,\delta}}]$. Similarly, it will extend to the interval $[t_0\wedge T_{{n,\delta}},2t_0\wedge T_{{n,\delta}}]$ and by iterating this argument, uniqueness
will be shown in $[0, T_{{n,\delta}}]$ (when $T_{{n,\delta}}=\infty$ we take this interval to be $[0,\infty)$).

 Then, since the processes do not explode, we can take the limit as $n\to \infty$,
to conclude uniqueness on $[0,T_\delta]$. Finally, we deduce that $X=\wX, Y=\wY$ on the interval 
$[0,T_0)$, where $T_0=\lim\uparrow  T_\delta$. 
Notice that one of the coordinates
has to be $0$ on the left of $T_0$, when $T_0$ is finite. 
Say that $X_{T_0-}=\wX_{T_0-}=0$. Since  $T_0$ is a predictable stopping time
the Poisson processes cannot jump at this time, which implies that $X_{T_0}=\wX_{T_0}=0$ and therefore
from the uniqueness starting from $0$ we conclude $X_{t}=\wX_{t}=0$ for all $t\ge T_0$, which also
implies that $Y_{t}=\wY_{t}=Y_{T_0-}$ for all $t\ge T_0$, showing the desired uniqueness.

\medskip \noindent Let us now prove \eqref{uni}. 

\medskip \noindent  In what follows we denote by
$$
\begin{array}{l}
\Delta X_s=X_{s}-\wX_{s}\ ;\  \Delta Y_s=Y_{s}-\wY_{s};\\
\Delta b_s=b_1(X_{s},Y_{s})-b_1(\wX_{s},\wY_{s})\ ;\, 
\Delta \ell_s^{1/2}=(\ell_1(X_{s},Y_{s}))^{1/2}-(\ell_1(\wX_{s},\wY_{s}))^{1/2};\\
\Delta \kappa_s=\kappa_1(X_{s-},Y_{s-})-\kappa_1(\wX_{s-},\wY_{s-})\ ; \ 
\Delta p_s=p_1(X_{s-},Y_{s-})-p_1(\wX_{s-},\wY_{s-})\ ; \\
\Delta u_s(\theta)=\ind_{\theta\le \kappa_1(X_{s-},Y_{s-})} p_1(X_{s-},Y_{s-})-
\ind_{\theta\le \kappa_1(\wX_{s-},\wY_{s-})} p_1(\wX_{s-},\wY_{s-}),
\end{array}
$$
and 
$$\Gamma_t=\int_0^{t} \int_{[0,\infty)^2} \Delta u_s(\theta)\, h(z) \, \nt^1(ds,d\theta, dz).$$
 We observe that
 \begin{equation}
 \label{decc}
 \Delta X_t=\int_0^t \Delta b_sds+ \int_0^{t} \Delta \ell_{s}^{1/2} dB^1_s+\Gamma_t+\int_0^\Td\int_{[0,\infty)^2}  \Delta u_s(\theta) (z -h(z)) N^1(ds,d\theta, dz)
 \end{equation}
\medskip \noindent 
and get
\begin{align}
\label{eq:e4}
\EE((\Delta X)^*_\Td)&\le \EE\left(\int_0^\Td |\Delta b_s|\, ds\right)+
\EE\left(\sup\limits_{0\le r\le \Td} \left| \int_0^{r} \Delta \ell_{s}^{1/2} dB^1_s\right|\right)\\
&\quad +\EE\left(\sup\limits_{0\le r\le \Td}|\Gamma_r|\right)
+\EE\left(\int_0^\Td\int_{[0,\infty)^2}  |\Delta u_s(\theta)| |z -h(z)| N^1(ds,d\theta, dz)\right).\nonumber
\end{align}

Consider $r(n)$ a common Lipschitz constant for all the functions
$b,\ell,\kappa,p,q$ on the interval $[0,n]$ and denote by 
$\K(n)=(2\mathbf{L}\, n+ \mathbf{A}+1) (r(n)+1)$,
where $\mathbf{L}, \mathbf{A}$ are given by the linear growth condition on {\bf (F3)}. In particular, $\K(n)$ serves
as a Lipschitz constant for all the coefficients of the system in the interval $[0,n]$, as well as a bound for these
functions on $[0,n]^2$.

\medskip
We introduce 
$$R_{t}= \EE((X-\wX)^*_\Td+(Y-\wY)^*_\Td) = \EE((\Delta X)^*_\Td+(\Delta Y)^*_\Td),$$
and
the first term in the RHS of \eqref{eq:e4} is clearly bounded by
\begin{equation}
\label{eq:first}
\EE\left(\int_0^\Td |\Delta b_s|\, ds\right)\le \K(n) \, t\, R_t.
\end{equation}

Let us bound the second term (Brownian term) in \eqref{eq:e4}. By definition for $s< \Td$, 
we have $X_s, Y_s, \wX_s, \wY_s \in [\delta, n]$ and therefore 
$a=\ell_1(X_s,Y_s)\ge \zeta, b=\ell_1(\wX_s,\wY_s)\ge \zeta$, where $\zeta=\zeta(\delta,n)>0$ is given by 
the ellipticity assumption {\bf (F5)}.
Now, for $a,b\ge \zeta$ we have $|\sqrt{a}-\sqrt{b}|\le \frac1{2\sqrt{\zeta}}|a-b|$ and we get from
\eqref{BDG} that
\begin{eqnarray}
\label{eq:second}
\EE\left(\sup\limits_{0\le r\le \Td} \left| \int_0^{r} \Delta \ell_{s}^{{}^{1/2}} dB^1_s\right|\right)
&\le &\CU \EE\left(\left(\int_0^{\Td}\left |\Big(\ell_1(X_s,Y_s)\Big)^{1/2}-
\Big(\ell_1(\wX_s,\wY_s)\Big)^{1/2}\,\right|^2\, ds\right)^{1/2}\right)
\nonumber\\
&\le & \frac{\CU}{2\sqrt{\zeta}} \EE\left(\left(\int_0^{\Td}\left |\ell_1(X_s,Y_s)-\ell_1(\wX_s,\wY_s)\right|^2\, ds\right)^{1/2}\right)\nonumber\\
&\le & \frac{\K(n) \CU}{2\sqrt{\zeta}}\, \sqrt{t}\, R_t
\end{eqnarray}
For the last term in \eqref{eq:e4}, we use that $0\le p\le 1$ and the triangular inequality 
$$
\begin{array}{l}
|\Delta u_s(\theta)|
\le
|\ind_{\theta\le \kappa_1(X_{s-},Y_{s-})} -
\ind_{\theta\le \kappa_1(\wX_{s-},\wY_{s-})}|+
\ind_{\theta\le \kappa_1(\wX_{s-},\wY_{s-})} |p_1(X_{s-},Y_{s-})- p_1(\wX_{s-},\wY_{s-})|.
\end{array}
$$
This implies that

\ben
\int_0^{\infty} |\Delta u_s(\theta)| \,d\theta&\le& |\kappa_1(X_{s-},Y_{s-})-\kappa_1(\wX_{s-},\wY_{s-})|+
\kappa_1(\wX_{s-},\wY_{s-}) |p(X_{s-},Y_{s-})- p(\wX_{s-},\wY_{s-})|\\
&\le &\K(n) \left(|X_{s-}-\wX_{s-}|+|Y_{s-}-\wY_{s-}|\right)
\een
and therefore
\begin{eqnarray}
\label{eq:elast}
\EE\left(\int_0^\Td |\Delta u_s(\theta)||z-h(z)|N^1(ds,d\theta, dz)\right)
&=&
\int_0^\infty |z-h(z)|\, \lambda_1(dz) \, \EE\left(\int_0^\Td\int_0^\infty |\Delta u_s(\theta)| \,ds d\theta \right)\nonumber\\
& \le & \K(n) \, t \int_0^\infty \!\!|z-h(z)|\, \lambda_1(dz) 
\; R_t. \nonumber
\end{eqnarray}

Let us now concentrate on the third term in \eqref{eq:e4}. We write
$$\Gamma^a_r=
\int_0^{r} \int_{[0,\infty)^2} \Delta u_s(\theta)\, h(z) \ind_{0\le z\le a} \, \nt^1(ds,d\theta, dz), \quad
\Gamma^{a\to}_r=\int_0^{r} \int_{[0,\infty)^2} \Delta u_s(\theta)\, h(z) \ind_{a< z} \, \nt^1(ds,d\theta, dz).$$
for $a\geq 0$ and  
$\sup\limits_{0\le s\le \Td} | \Gamma_s|  \le \sup\limits_{0\le s\le \Td} | \Gamma^a_s|+
\sup\limits_{0\le s\le \Td} | \Gamma^{a\to}_s|$.
Using now \eqref{BDG} and \eqref{ineq:BDG-Lenglart} and $\ind_{a<z} |d\nt^1| \leq \ind_{a<z}d(N^1+\nu^1)$, we get 
\begin{align} 
\label{eq:ethird}
&\EE\left(\sup\limits_{0\le s\le \Td} | \Gamma_s|\right) \nonumber \\ &\qquad \le  \EE\left(\sup\limits_{0\le s\le \Td} | \Gamma^a_s|\right)+
\EE\left(\int_0^\Td\int_0^{\infty} |\Delta u_s(\theta)| \,h(z) \ind_{a<z} \, d(N^1+\nu^1) \right)\nonumber \\
&\qquad \le \CU\EE\left([\Gamma^a,\Gamma^a]^{1/2}_\Td\right)+
2 \int_a^\infty\!\!\!\! h(z)\, \lambda_1(dz)\; \EE\left(\int_0^\Td \int_0^\infty |\Delta u_s(\theta)| \, ds d\theta  \right)\\
&\qquad \le 3\CU \EE\left(\langle \Gamma^a, \Gamma^a \rangle^{1/2}_\Td\right)+
2\K(n) \int_a^\infty\!\!\!\! h(z)\, \lambda_1(dz)\; R_t.
\end{align}

It remains to estimate $\EE\left(\langle \Gamma^a, \Gamma^a \rangle^{1/2}_\Td\right)$.
If we denote by $W_1(a)=
\int_{[0,a]} h^2(z) \lambda_1(dz)$ and $W_1=W_1(\infty)$, then for $0<a\le 1$, 
to be fixed later on, we get
\begin{align*}
\EE\left(\langle \Gamma^a, \Gamma^a\rangle^{1/2}_\Td\right)&=\sqrt{W_1(a)}\;
\EE\left(\left(\int_0^\Td \int_0^\infty |\Delta u_s| \,d\theta ds\right)^{1/2}\right)\\
&=\frac{\sqrt{W_1(a)}}{\sqrt{W_1}}\;\sqrt{W_1}\;
\EE\left(\left(\int_0^\td  \int_0^\infty |\Delta u_s| \,d\theta ds\right)^{1/2}\right)\\
&=\frac{\sqrt{W_1(a)}}{\sqrt{W_1}}\;
\EE\left(\langle \Gamma, \Gamma\rangle^{1/2}_\Td\right)\\
&\le \frac{\sqrt{W_1(a)}}{\sqrt{W_1}} \,\cb^{\hspace{-0.1cm}-1}
\left(\EE\left(\sup\limits_{0\le s\le \Td}\left|\Gamma_s \right|\right)+\NvDelta\right),
\end{align*}
where we applied  \eqref{ineq:BDG-Lenglart} to $(\Gamma_s)_s$.
Here  $\cb^{\hspace{-0.1cm}-1}$ is a finite constant, and obviously 
since $\NvDelta$ is a bound for $h$, then $\NvDelta$ is a bound for the jumps of $\Gamma$.

\noindent It remains to remark from \eqref{decc} that 
\begin{align}
\label{eq:e3}
\EE\left(\sup\limits_{0\le r\le \Td}\left| \Gamma_r\right|\right)
&\le \EE\left((\Delta X)^*_{T_{n,\delta}}\right)+\EE\left(\int_0^{T_{n,\delta}} |\Delta b_s|\, ds\right)+
\EE\left(\sup\limits_{0\le r\le T_{n,\delta}} \Big| \int_0^{r} \Delta \ell^{{}^{1/2}}_s dB^1_s\Big|\right)\nonumber\\
&\hspace{0.3cm}+\EE\left(\int_0^{T_{n,\delta}}\int_{[0,\infty)^2} |\Delta u_s(\theta)| |z-h(z)| N^1(ds,d\theta, dz)\right)=\mathscr{K}<+\infty.
\end{align}

\noindent   Coming back to \eqref{eq:ethird} we get 
$$\EE\left(\sup\limits_{0\le s\le \Td} | \Gamma_s|\right)\le
 \frac{\sqrt{W_1(a)}}{\sqrt{W_1}} \,3\CU\cb^{\hspace{-0.1cm}-1}
(\mathscr{K}+\NvDelta)+ 2\K(n) \int_a^\infty\!\!\!\!h(z)\, \lambda_1(dz)\; 
\int_0^t R_s ds.$$
\medskip \noindent
Finally, adding all the estimates in  \eqref{eq:e4}, we  obtain the following inequality
$$
\begin{array}{ll}
\EE((\Delta X)^*_\Td)\hspace{-0.2cm}&\le \beta_1(a) \EE\left((\Delta X)^*_\Td\right) 
+\beta_1(a)\NvDelta+\gamma_1(a) \int_0^t R_s\, ds+\rho_1(a,t)R_t,
\end{array}
$$
with
\begin{align*}
&\beta_1(a)=\frac{\sqrt{W_1(a)}}{\sqrt{W_1}} \,3\CU\cb^{\hspace{-0.1cm}-1}; \, 
\gamma_1(a)=2\K(n) \int_a^\infty h(z)\, \lambda_1(dz),\\
&\rho_1(a,t)=\K(n) \left(t+t\int_0^\infty |z-h(z)|\, \lambda_1(dz)+
\frac{\CU}{2\sqrt{\zeta}}\, \sqrt{t}\right)(1+\beta_1(a)).
\end{align*}

In a similar way, we get the upper bound for $\EE((\Delta Y)^*_\Td)$. We call 
$\beta_2, \gamma_2, \rho_2$ the corresponding quantities.
Then, summing up these upper bounds gives the following upper
bound for $R_t$.
$$
R_t\le (\beta_1(a)\vee \beta_2(a))\,  R_t+
\NvDelta\beta(a)+\gamma(a) \int_0^t R_s \, ds+\rho(a,t)R_t,
$$
with $\beta=\beta_1+\beta_2, \gamma=\gamma_1+\gamma_2, \rho=\rho_1+\rho_2$.

We first choose $0<a_0<1$ such that for all $a\le a_0$, we have 
$\beta(a)\le 1/4$, and we choose $0<t_0=t_0(n)$ such that 
$$
\begin{array}{ll}
\rho(1,t_0)\, &=\, \K(n)\left(t_0+2 t_0\int_0^\infty\! |z-h(z)|\, \lambda_1(dz)+
\frac{\CU}{2\sqrt{\zeta}}\, \sqrt{t_0}\right)(1+\beta_1(1))\\
&\quad+\K(n) \left(t_0+2 t_0\int_0^\infty \! |z-h(z)|\, \lambda_2(dz)+
\frac{\CU}{2\sqrt{\zeta}}\, \sqrt{t_0}\right)(1+\beta_2(1)) \,  \le   \, 1/4.
\end{array}
$$
Hence, for all $a\le a_0, t\le t_0$, we get
$R_t\le \frac12  R_t+\NvDelta\beta(a)+\gamma(a)\, \int_0^t R_s \, ds$, and a fortiori it holds
\begin{equation}
\label{eq:4} 
R_t\le 2\NvDelta\beta(a)+2\gamma(a)\, \int_0^t R_s \, ds.
\end{equation}
Gronwall's inequality shows that, for
all $0<a\le a_0, t\le t_0$
\ben
R_t
&=&6\, \NvDelta \, \CU\cb^{\hspace{-0.1cm}-1}\left(\sqrt{\frac{W_1(a)}{W_1}}+
\sqrt{\frac{W_2(a)}{W_2}}\right)e^{2t \gamma(a)} \\
&\le &  12\, \NvDelta\frac{\CU\cb^{\hspace{-0.1cm}-1}}{\sqrt{W_1}\,\wedge \,\sqrt{W_2}}
\left(\int_0^a h(z)^2\, \lambda(dz)\,  e^{8 \K(n)t\, \int	\limits_a^\infty h(z)\, \lambda(dz)}\right)^\frac12\\
&\le & A\left(\int_0^a h^2(z)\, \lambda(dz)\,  e^{8\K(n)t_0 \, \int_a^\infty z\, \lambda(dz)}\right)^\frac12,
\een
where  $A=12 \, \NvDelta \, \frac{\CU\cb^{\hspace{-0.1cm}-1}}{\sqrt{W_1\,\wedge \,W_2}}$.
Hence, if we also assume that $8\K(n)t_0 \le \epsilon_0$, we have for all $t\le t_0$ 
$$
\EE((X-\wX)^*_\Td+(Y-\wY)^*_\Td)\le A \liminf_{a\to 0} \left(\int_0^a h(z)^2\, \lambda(dz)\,  
e^{\epsilon_0 \int_a^\infty h(z)\, \lambda(dz)}\right)^\frac12=0.
$$
This result was our aim and as previously detailed, uniqueness is then proved under  \eqref{eq:extra}.

\medskip

(ii) Now, we relax the extra integrability condition \eqref{eq:extra}. 
We truncate the Poisson processes as follows
$$
N^{i,D}(dt,d\theta,dz)=\ind_{z\le D}\, N^i(dt,d\theta,dz), \; i=1,2,
$$
where now the intensities are 
$d\widehat \nu_i^{D}(dt,d\theta,dz)=\ind_{z\le D} \, dt\, d\theta\,\lambda_i(dz)$.
In particular, we have 
$\widehat\lambda_i^{D}(dz)=\ind_{z\le D} \, \lambda_i(dz)$, which satisfy {\bf (F6)} and the
extra condition  of part $(i)$ is satisfied:
$$
\int_0^{\infty} (z^2\wedge z) \widehat\lambda_i^D(dz)=
\int_{0}^{D} (z^2\wedge z)\, \lambda(dz)<\infty.
$$
We consider the associated drift term, where compensation has been truncated :
$$b_i^D(x,y)=b_i(x,y)-  \kappa_i(x,y)p_i(x,y)\, 
 \int_{(D,\infty)} h(z)\lambda(dz)$$
With these truncated Poisson processes, consider the analogue of \eqref{eq:e1}
\begin{eqnarray}
\label{eq:e1'}
&\hX_t&\hspace{-0.3cm}=x_0+\int_0^t b_1^D(\hX_s,\hY_s) ds + 
\int_0^t \sqrt{\ell_1 (\hX_s,\hY_s)} \, dB^1_s\nonumber\\
&&\hspace{0.5cm}+\int_0^t \int_{[0,\infty)^2} \ind_{\theta\le \kappa_1(\hX_{s-},\hY_{s-})}
\,  p_1(\hX_{s-},\hY_{s-})\, h(z)\,  \nt^{1,D}(ds, d\theta,dz)\nonumber\\
&&\hspace{0.5cm}+\int_0^t \int_{[0,\infty)^2} \ind_{\theta\le \kappa_1(\hX_{s-},\hY_{s-}) }\, 
p_1(\hX_{s-},\hY_{s-}) \,(z-h(z)) \, N^{1,D}(ds, d\theta,dz);\nonumber\\
&\hY_t&\hspace{-0.3cm}=y_0+\int_0^t b_2^D(\hX_s,\hY_s) ds 
+\int_0^t \sqrt{\ell_1 (\hX_s,\hY_s)} \, dB^2_s\nonumber\\
&&\hspace{0.5cm}+ \int_0^t \int_{[0,\infty)^2} \ind_{\theta\le \kappa_2(\hX_{s-},\hY_{s-})} 
 \, p_2(\hX_{s-},\hY_{s-})\, h(z)\,  \nt^{2,D}(ds, d\theta,dz)\nonumber\\
&&\hspace{0.5cm}+\int_0^t \int_{[0,\infty)^2} \ind_{\theta\le \kappa_2(\hX_{s-},\hY_{s-})}  \, 
p_2(\hX_{s-},\hY_{s-}) \, (z-h(z)) \, N^{2,D}(ds, d\theta,dz).
\end{eqnarray}
 We claim that if $(X,Y)$, $(\wX,\wY)$ are two solutions of \eqref{eq:e1}, then they are also solutions of 
\eqref{eq:e1'}, on the interval $[0,\tau^D)$, where $\tau^D$ is the first time when
the point measure induces a jump $z$ larger than $D$:
$$\tau^D=\inf\left\{t >0 :\int_0^t \int_{[0,\infty)^2} \ind_{\theta\le \kappa_1(X_{s-},Y_{s-}), \, z> D}  N^{1}(ds, d\theta,dz)+ \ind_{\theta\le \kappa_2(X_{s-},Y_{s-}), \, z>D}  N^{2}(ds, d\theta,dz)\, > \,0\right\}$$
Indeed, 
  $\ind_{s<\tau^D,\,  \theta\le \kappa_i(X_{s-},Y_{s-}) }\ind_{z>D} N^{i}(ds, d\theta, dz) =0$ and 
$$\ind_{s<\tau^D, \, \theta\le \kappa_i(X_{s-},Y_{s-}) } \, N^{i}(ds, d\theta,dz)=
\ind_{s<\tau^D,\, \theta\le \kappa_i(X_{s-},Y_{s-}) } \, N^{i,D}(ds, d\theta,dz),$$
while $b_i^D-b_i$ is the correction of the drift coming from the compensation of  $N^{1,D}-N^1$.\\

The first part $(i)$ then ensures that $(X,Y)$ and $(\wX,\wY)$ coincide up to time $\tau^D$.
Writing $S_n=S_n(X,Y,\wX,\wY)$ the first time when either $X$ or $Y$ or $\wX$ or $\wY$ goes beyond $n$, we observe that
$$\{\tau^D< t\wedge S_n\} \subset \cup_{i\in\{1,2\}} \{ N^i ([0,t]\times [0,\sup \kappa_i([0,n]^2)] \times [D,\infty))>0\}.$$
Besides,   for each $n\in\mathbb N$ and $t>0$ the probability  of the event of the right hand side goes to $0$ as $D\rightarrow \infty$. 
Letting $n$ and then $D$ go to infinity ensures uniqueness up to explosion time $\mathbb T_e$.
 The proof is completed.
\end{proof}

\appendix 
\section{Hypotheses (H1) and (H2)}
In this appendix, we recall the framework introduced in \cite{BCM} Section 2, adapted to our setting. 

\me Let  ${\cal X}$   be   the bounded subset $(0,1]^2$ of $\mathbb{R}^2$ and $\mathcal U=[-1,1]^2$. 

\me 
  For any $N\geq 1$, we consider 
a discrete time ${\cal X}$-valued   Markov chain $(X^N_{k} :  k\in \mathbb{N})$ with increments $X^N_{k+1} - X^N_{k}$ taking values in $\mathcal U$.   Let $(v_{N})_{N}$ be a given sequence of positive real numbers going to infinity when $N$ tends to infinity. For $x\in {\cal X}$, we define\be\label{GNx}
{\cal G}_{x}^N(H) =  v_{N}\, \E\big(H(X^N_{k+1} - X^N_{k})\,| \, X^N_{k}= x\big)=v_N\, \E\big(H(X^N_{1} - X^N_{0})\,| \, X^N_{0}= x\big),\ee
for real valued bounded measurable functions $H$ defined on ${\cal U}$. 

\bi  We first observe that   Hypothesis (H0) in \cite{BCM} is obviously satisfied since the state space is bounded.

\me 
We   introduce the functional space
$$C_{b,0}^2 = C^2_{b,0}({\cal U}, \mathbb{R})=  \left\{ H\in C_{b}( {\cal U}, \mathbb{R}) \, :  \, H(u) = \sum_{i=1}^2 \alpha_{i}u_{i} + \sum_{i,j=1}^2 \beta_{i,j} u_{i} u_{j}+ o(|u|^2), \, \alpha_{i}, \beta_{i,j} \in \R \right\}$$
and here
  the  {\bf{specific function}}  $h$   is the two dimensional identity function
  \be \label{truncation} h = (h^1,   h^2) \in  (C^2_{b,0})^2\ ;\ h^i(u)=  u_{i} \quad (i=1, 2).\ee

 
 \bi {\bf Hypotheses (H1)}
 There exists a functional space $\mathcal H$ such that 
 \begin{enumerate}
 \item   \emph{${\cal  H}$ is a subset of  $C^2_{b,0}$ and $\, h^i, h^ih^j\, \in Vect( {\cal H})$  for $i,j = 1, \ldots, 2$.}
\item  \emph{For any  $ g\in C_{c}({\cal U}, \mathbb{R})$ with $g(0) = 0$, there exists a sequence $(g_{n})_{n}  \in C^2_{b,0}$ such that $\,\lim_{n \to \infty} \|g-g_{n}\|_{\infty, {\cal U} }= 0$
and $|h|^2\, g_n \in Vect({\cal H})$.} 

 \item    \emph{There exists a family of real numbers $\,({\cal G}_{x}( H );  x\in {\cal X}, H \in {\cal  H})$ such that for any $H \in {\cal  H}$}, 
 \ben
(i) && \qquad \lim_{N\to \infty}\sup_{x\in {\cal X}} \left| {\cal G}_{x}^N(H) - {\cal G}_{x}(H) \right| = 0. \qquad \qquad\\
(ii) &&  \qquad 
\sup_{x\in {\cal X}} \left|  {\cal G}_{x}(H) \right|  < + \infty.
\een
\end{enumerate}

\label{statgen}


 \begin{thm}
\label{thm-tightness} Assume that the sequence $(X_0^N)_N$  is tight in $\overline{{\cal X}}$
and ${(H1)}$  hold. Then the  sequence of processes $\,(X^N_{[v_{N}.]}, N\in \mathbb{N})$ is tight in $\mathbb{D}([0,\infty), \overline{{\cal X}})$.
\end{thm}

\me 
We observe that  that  for any 
$H\in C^2_{b,0}$,  there exists a unique decomposition of the form
\be
\label{decomp} H = \sum_{i=1}^2 \alpha_{i}(H) h^{i} + \sum_{i,j=1}^2 \beta_{i,j}(H)  h^{i}\, h^{j}+\overline{H},\ee where
 $\overline{H}(u_{1},u_{2})= o(|(u_{1},u_{2})|^2)$ is a continuous and bounded function and $\alpha_{i}(H)$, $\beta_{i,j}(H)$, $i,j=1 \cdots 2$ are real coefficients and $\beta$ is a symmetric matrix.

\me 
 The next hypothesis {(H2)} in addition to {(H1)}  is sufficient to get the identification
  of the limiting values by their  semimartingale characteristics, and  then their representation  as solutions of a stochastic differential equation. 

\bi {\bf Hypotheses (H2)}   
 \begin{enumerate}
 \item    \emph{For any  $\,H\in {\cal H}$, the map $  x\in  {\cal X} \rightarrow{\cal G}_{x}(H) $ is continuous and extendable by continuity to $ \overline {\cal X}.$} 
\item  \emph{For
any $\,x\in \overline{{\cal X}}$ and any $\,H\in {\cal H}$},
\be
\label{h2}
 {\cal G}_{x}(H) =  \sum_{i=1}^2 \alpha_{i}(H) b_{i}(x) + \sum_{i,j=1}^2  \beta_{i,j}(H)c_{i,j}(x)   + \int_{V} \overline{H}(K(x,v)) \mu(dv), 
 \ee
 \end{enumerate}
 \emph{where
 \begin{itemize}
\item[i)] 
  $\alpha_{i}$, $\beta_{i,j}$ and $\overline{H}$ have been defined in \eqref{decomp}, 
  \item[ii)]  $b_{i}$ and $c_{i,j}$ are measurable functions defined  on $\overline{{\cal X}}$, 
  \item[iii)]  $V$ is a Polish space, $\mu$ is a $\sigma$-finite positive measure  on $V$, $K$ is a measurable function from $\overline{{\cal X}}\times V$ with values in ${\cal U}$,  $ \int_{V}   1\wedge\vert K(.,v)\vert^2  \mu(dv)<+\infty$ and 
$$c_{i,j}(x)=\sum_{k=1}^4\sigma_{i,k}(x)\sigma_{j,k}(x) +\int_{V} K_{i}(x,v)K_{j}(x,v) \mu(dv),$$
where $\sigma_{i,k}(x)$ are measurable functions  defined  on $\overline{{\cal X}}$ for $1\leq i\leq 2$ and $1\leq k\leq 4$. 
\end{itemize}}


The main general result in \cite{BCM} yields the following statement here. A slight adaptation is needed
 since here the dimension  $2$ of the process $X$  differs from the dimension  $4$ of the brownian motion involved in the representation (one can also  consider a $4$-dimensional process by adding two coordinates identically null to match the precise framework of \cite{BCM}).

\bi
\begin{thm}
 \label{identification}  
  If the sequence $(X_0^N)_N$  is tight in $\overline{{\cal X}}$ 
and  {(H1)} and  {(H2)} hold  then any  limiting value of $\,(X^N_{[v_{N}.]}, N\in \mathbb{N})$  is a   semimartingale    solution of the  stochastic differential system
 \be
 \label{eds}
 X_{t}&=&X_{0} + \int_{0}^t b(X_{s}) ds + \int_{0}^t \sigma(X_{s}) dB_{s} + \int_{0}^t\int_{V} K(X_{s-},v) \tilde N(ds, dv),
 \ee
where  
 $X_{0}\in \overline{{\cal X}}$ and $B$ is a $4$-dimensional Brownian motion and
    $N$ is a Poisson Point  measure on $\R_{+}\times V$ with intensity $ds \mu(dv)$. Moreover  $X_{0}$, $B$, $N$ are independent and $\tilde N$ is the   compensated martingale measure of $N$.
    \end{thm}

\section{Hypothesis  ${\bf (F6)}$}

In this appendix, we shall study more closely hypothesis ${\bf (F6)}$.

\begin{lem} Assume $\lambda_1$ satisfies ${\bf (F6)}$, and $\lambda_2$
satisfies $\int (z\wedge 1)\, \lambda_2(dz)<\infty$. Then
$\lambda=\lambda_1+\lambda_2$ satisfies ${\bf (F6)}$.
\end{lem}
\begin{proof} If $\lambda_1$ also satisfies $\int (z\wedge 1)\, \lambda_1(dz)<\infty$, we
have for $\lambda=\lambda_1+\lambda_2$ and any $a\le 1, \epsilon>0$
$$
e^{\epsilon\int_{(a,1]} z\,\lambda(dz)}\int_0^a z^2\,\lambda(dz)\le 
e^{\epsilon\int_{[0,1]} z\,\lambda(dz)}\int_0^a z^2\,\lambda(dz),
$$
which converges to $0$ as $a$ converges to $0$. Hence, $\lambda$ satisfies ${\bf (F6)}$.

\smallskip

So, for the rest of the proof we shall assume that $\int_0^1 z\, \lambda_1(dz)=+\infty$. 
In what follows we denote by $K=\int_{[0,1]} z\,\lambda_2(dz)<\infty$. We define inductively
$c_0=1$ and given $c_n$ we consider $0<c_{n+1}<c_n$ characterized by
$$
c_{n+1}=\sup\left\{0\le c<c_n:\, \int_{[c,c_n)} z\, \lambda_1(dz)\ge 1\right\}.
$$
Now, since $\lambda_1$ satisfies ${\bf (F6)}$ there exists a sequence $(a_k)_k\subset [0,1]$ such that 
$a_k\downarrow 0$ and 
$$
r(a_k)=e^{\epsilon_0\int_{(a_k,1]} z\lambda_1(dz)}\int_0^{a_k} z^2\,\lambda_1(dz)\to 0.
$$
Consider for every $k$ the unique $c_{n_k}$ such that $a_k\in [c_{1+n_k},c_{n_k})$. We consider two
possible situations: 
\begin{itemize}
\item[(i)]  $\int_{[c_{1+n_k},a_k]} z\lambda_1(dz)\le 1/2$;
\item[(ii)] $\int_{[c_{1+n_k},a_k]} z\lambda_1(dz)> 1/2$.
\end{itemize}
In the first case we have $\int_{(c_{2+n_k},a_k]} z\, \lambda_1(dz)\le 3/2$ and
$\int_{[c_{2+n_k},a_k]} z\, \lambda_1(dz)\ge 1$. On the one hand
$$
\begin{array}{ll}
r(a_k)\hspace{-0.3cm}&=e^{\epsilon_0\int_{(a_k,1]} z\lambda_1(dz)}\int_0^{a_k} z^2\,\lambda_1(dz)\ge
e^{\epsilon_0\int_{(a_k,1]} z\lambda_1(dz)}\int_{[c_{2+n_k},a_k]} z^2\,\lambda_1(dz)\\
\\
\hspace{-0.3cm}&\ge  c_{2+n_k} e^{\epsilon_0\int_{(a_k,1]} z\lambda_1(dz)} \int_{[c_{2+n_k},a_k]} z\,\lambda_1(dz)
\ge c_{2+n_k} e^{\epsilon_0\int_{(a_k,1]} z\lambda_1(dz)} \\
\\
\hspace{-0.3cm}&\ge e^{\epsilon_0\int_{(c_{2+n_k},1]} z\lambda_1(dz)-\frac32\epsilon_0}\; c_{2+n_k}.
\end{array}
$$
With this estimation we obtain for $d_k=c_{2+n_k}$
$$
\begin{array}{ll}
e^{\epsilon_0\int_{(d_k,1]} z\lambda_1(dz)} \int_0^{d_k} z^2 \lambda_2(dz)
\hspace{-0.3cm}&\le c_{2+n_k} e^{\epsilon_0\int_{(c_{2+n_k},1]} z\lambda_1(dz)} \int_0^{d_k} z \lambda_2(dz)\;\\
\\
\hspace{-0.3cm}&\le  K e^{\frac32\epsilon_0}\; r(a_k).
\end{array}
$$
On the other hand 
$$
\begin{array}{ll}
e^{\epsilon_0\int_{(d_k,1]} z\lambda_1(dz)} \int_0^{d_k} z^2 \lambda_1(dz)
\hspace{-0.3cm}&\le e^{\epsilon_0\int_{(a_k,1]} z\lambda_1(dz)+\frac32 \epsilon_0} \int_0^{a_k} z^2 \lambda_1(dz)\;\\
\\
\hspace{-0.3cm}&\le e^{\frac32\epsilon_0}\; r(a_k).
\end{array}
$$
This gives the bound
$$
e^{\epsilon_0\int_{(d_k,1]} z\lambda(dz)} \int_0^{d_k} z^2 \lambda(dz)\le
e^{\epsilon_0 (K+\frac32)}( K+1)\; r(a_k).
$$
In the second case we have $\int_{(c_{1+n_k},a_k]} z\lambda_1(dz) < 1$, by the definition of $c_{1+n_k}$, and 
therefore
$$
\begin{array}{ll}
r(a_k)\hspace{-0.3cm}&=e^{\epsilon_0\int_{(a_k,1]} z\lambda_1(dz)}\int_0^{a_k} z^2\,\lambda_1(dz)\ge
e^{\epsilon_0\int_{(a_k,1]} z\lambda_1(dz)}\int_{[c_{1+n_k},a_k]} z^2\,\lambda_1(dz)\\
\\
\hspace{-0.3cm}&\ge c_{1+n_k} e^{\epsilon_0\int_{(a_k,1]} z\lambda_1(dz)} \int_{[c_{1+n_k},a_k]} z\,\lambda_1(dz)
\ge \frac12 c_{1+n_k} e^{\epsilon_0\int_{(a_k,1]} z\lambda_1(dz)} \\
\\
\hspace{-0.3cm}&\ge \frac12 c_{1+n_k} e^{\epsilon_0\int_{(c_{1+n_k},1]} z\lambda_1(dz)-\epsilon_0}\; .
\end{array}
$$
Similarly as before, we take $d_k=c_{1+n_k}$ which gives
$$
\begin{array}{ll}
e^{\epsilon_0\int_{(d_k,1]} z\lambda_1(dz)} \int_0^{d_k} z^2 \lambda_2(dz)
\hspace{-0.3cm}&\le c_{1+n_k} e^{\epsilon_0\int_{(c_{1+n_k},1]} z\lambda_1(dz)} \int_0^{d_k} z \lambda_2(dz)\;\\
\\
\hspace{-0.3cm}&\le  2K e^{\epsilon_0}\; r(a_k).
\end{array}
$$
Again, we have
$$
\begin{array}{ll}
e^{\epsilon_0\int_{(d_k,1]} z\lambda_1(dz)} \int_0^{d_k} z^2 \lambda_1(dz)
\hspace{-0.3cm}&\le e^{\epsilon_0\int_{(a_k,1]} z\lambda_1(dz)+\epsilon_0} \int_0^{a_k} z^2 \lambda_1(dz)\;\\
\\
\hspace{-0.3cm}&\le e^{\epsilon_0}\; r(a_k).
\end{array}
$$
which allows us to show
$$
e^{\epsilon_0\int_{(d_k,1]} z\lambda(dz)} \int_0^{d_k} z^2 \lambda(dz)\le
e^{\epsilon_0 (K+1)}( 2K+1)\; r(a_k).
$$
We summarize these estimations in both cases as
$$
e^{\epsilon_0\int_{(d_k,1]} z\lambda(dz)} \int_0^{d_k} z^2 \lambda(dz)\le 
e^{\epsilon_0 (K+\frac32)}( 2K+1)\; r(a_k)
$$
The result follows by noticing that $(d_k)_k$ converges to $0$.
\end{proof}

\begin{rem} Notice that if $\int (z\wedge 1)\, \lambda(dz)<\infty$ then both 
$\int (z\wedge 1)\, \lambda_1(dz)<\infty, \int (z\wedge 1)\, \lambda_2(dz)<\infty$
and a fortiori both $\lambda_1,\lambda_2$ fullfill ~${\bf (F6)}$. 
Moreover,  $\lambda=\lambda_1+\lambda_2$ satisfies ${\bf (F6)}$. 

\smallskip

It is also direct to show that if  $\lambda$ satisfies ${\bf (F6)}$, then both $\lambda_1,\lambda_2$ 
fulfill ~${\bf (F6)}$. Nevertheless, it is not true that if both $\lambda_1,\lambda_2$ satisfies 
${\bf (F6)}$ then $\lambda$ satisfies ${\bf (F6)}$. This makes the previous Lemma more interesting.
\end{rem}
We now give sufficient conditions to have hypothesis ${\bf (F6)}$. 
Assume that $\lambda(dz)=\frac{f(z)}{z^2} dz$, with $f\ge 0$ and $\int_0^1 f(z) dz<\infty$, 
$\int_{0+}\frac{f(z)}{z}\, dz=\infty$. 
After taking logarithm, condition ${\bf (F6)}$ holds if 
\begin{equation}
\label{eq:8}
\sup\limits_{0<a<a_0}\frac{\int_a^{a_0} \frac{f(z)}{z} dz}{-\log(\int_0^a f(z) \, dz)}<\infty,
\end{equation}
for some small $a_0$. This condition is satisfied if
\begin{equation}
\label{eq:9}
r=\sup\limits_{0<a<a_0}\frac{\int_0^a f(z) dz}{a}<\infty.
\end{equation}
Indeed, for all small $z$ we have $\int_0^z f(u) du\le r z$ and therefore
$$
\int_a^{a_0} \frac{f(z)}{z} dz\le r \int_a^{a_0} \frac{f(z)}{\int_0^z f(u) du} dz=r\left(\log\left(\int_0^{a_0} f(u) du\right)-
\log\left(\int_0^{a} f(u) du\right)\right).
$$
Hence, if $a_0$ is small, we have $\int_0^{a_0} f(u) du<1$ and then
$$
0\le \frac{\int_a^{a_0} \frac{f(z)}{z} dz}{-\log(\int_0^a f(z) \, dz)}\le r\left(\frac{\log\left(\int_0^{a_0} f(u) du\right)}{-\log(\int_0^a f(z) \, dz)}+1\right)\le r.
$$
which shows \eqref{eq:8}.

Notice that \eqref{eq:9} is satisfied if $f$ is bounded near $0$. For example $f=1$, which gives 
$\lambda(dz)=\frac{1}{z^2} dz$.
Clearly $f(z)=-\log(z)$, for $z$ small, does not satisfies \eqref{eq:9}. It is quite direct to show that it does not satisfies 
\eqref{eq:8} nor \eqref{hyp:H}.

\end{document}